\documentclass{article}
\usepackage{amsthm,amsmath,amssymb}
\RequirePackage[numbers]{natbib}
\RequirePackage[colorlinks,citecolor=blue,urlcolor=blue]{hyperref}

\RequirePackage[OT1]{fontenc}
\RequirePackage[numbers]{natbib}
\RequirePackage[colorlinks,citecolor=blue,urlcolor=blue]{hyperref}
\usepackage{amsthm,amsmath,amssymb}

\usepackage{graphicx, subfigure}
\usepackage[top=1in,bottom=1in,left=1in,right=1in]{geometry}
\usepackage{amsfonts, amsmath, amssymb, amsthm, constants}

\usepackage{dsfont}
\usepackage{mathrsfs}
\usepackage{verbatim}

\usepackage{url}
\usepackage{color}
\usepackage{amsfonts}
\numberwithin{equation}{section}
\theoremstyle{plain}
\newtheorem{thm}{Theorem}[section]
\newtheorem{Theorem}{Theorem}[section]
\newtheorem{Proposition}[thm]{Proposition}
\newtheorem{Lemma}[thm]{Lemma}
\newtheorem{Corollary}[thm]{Corollary}
\newtheorem{Assumption}[thm]{Assumption}
\newtheorem{Definition}[thm]{Definition}
\theoremstyle{remark}

\newcommand{\ee}{\mathbb E}
\newcommand{\pp}{\mathbb P}

\newcommand{\eps}{\varepsilon}
\newcommand\independent{\protect\mathpalette{\protect\independenT}{\perp}}
\def\independenT#1#2{\mathrel{\rlap{$#1#2$}\mkern2mu{#1#2}}}

\definecolor{jens}{rgb}{0,0.4,0.8}

\def\bR{\mathbb R}
\def\bE{\mathbb E}

\def\argmin{\mathop{\rm arg\,min}}

\def\rank{{\rm rank}}

\def\Var{\mathop{\rm Var}\nolimits}

\begin{document}
\title{Adaptive confidence sets for matrix completion}
%
\maketitle

\author{\begin{center}
Alexandra Carpentier, \textit{ Universit\"at Potsdam\footnote{Institut f\"ur Mathematik,  carpentier@maths.uni-potsdam.de}}\\ \vskip 0.25 cm Olga Klopp, \textit{ University Paris Ouest\footnote{MODAL’X,  kloppolga@math.cnrs.fr}}\\ \vskip 0.25 cm Matthias L\"offler and Richard Nickl, \textit{ University of Cambridge\footnote{Statistical Laboratory, 
 Centre for Mathematical Sciences, m.loffler@statslab.cam.ac.uk, r.nickl@statslab.cam.ac.uk}}\end{center} }

\maketitle
	\begin{abstract}
			In the present paper we study the problem of existence of honest and adaptive confidence sets for matrix completion. We consider two statistical models: the trace regression model and the Bernoulli model.
			In the trace regression model, we show that  honest confidence sets that adapt to the unknown rank of the matrix exist even when the error variance is unknown. Contrary to this, we prove that in the Bernoulli model, honest and adaptive confidence sets exist only when the error variance is known a priori. In the course of our proofs we obtain bounds for the minimax rates of certain composite hypothesis testing problems arising in low rank inference.
				\end{abstract}
\maketitle
\section{Introduction}
\noindent
In matrix completion we observe $n$ noisy entries of a  data matrix $M=(M_{ij})\in \mathbb{R}^{m_{1}\times m_{2}}$, and we aim at doing inference on $M$. In a typical situation of interest, $n$ is much smaller than  $m_1m_2$, the total number of entries. 
This problem arises in many applications such as recommender systems and collaborative filtering \cite{BennetLanning, Goldberg}, genomics \cite{Chi} or sensor localization \cite{Singer}.
Two statistical models have been proposed in the matrix completion literature: the \emph{trace-regression model} (e.g. \cite{CaiZhou, KloppTrace, Koltchinskiietal, NegahbanWainwright, RohdeTsybakov}  ) and the \emph{`Bernoulli model'} (e.g. \cite{CandesRecht, Chatterjee, KloppThresh}). 

\smallskip

In the \textit{trace-regression model} we observe $n$ pairs $(X_i,Y_i^{tr})$ satisfying
\begin{equation}\label{trace_model}
Y_{i}^{tr}=\left\langle X_i,M\right\rangle + \epsilon_i = \mathrm{tr}(X_{i}^{T}M)+\epsilon_{i},~~ \:i=1,\dots,n,
\end{equation} 
where $(\epsilon_{i})$ is a noise vector. The random matrices $X_{i}\in \mathbb{R}^{m_1\times m_2}$ are independent of the $\epsilon_i$'s, chosen uniformly at random from the set
\begin{equation}\label{basisUSR}
\mathcal{B} = \left \{e_j(m_1)e_k^{T}(m_2),~1\leq j\leq m_1, ~1\leq k\leq m_2\right \},
\end{equation}
were the $e_j(s)$ are the canonical basis vectors of $\bR^{s}$. In this model $Y_i^{tr}$ returns the noisy value of the entry of $M$ corresponding to the random position $X_i$. 

\smallskip

In the \textit{Bernoulli model} each entry of $M+E$, where $E = (\epsilon_{ij}) \in \mathbb{R}^{m_1\times m_2}$ is a matrix of random errors, is observed independently of the other entries with probability $p=n / (m_1 m_2)$. More precisely, if $n \le m_1 m_2$ is given and $B_{ij}$ are i.i.d.~Bernoulli random variables  of parameter $p$ independent of the $\epsilon_{ij}$'s, we observe
\begin{equation} \label{Bernoulli_model}
Y_{ij}^{Ber}=B_{ij}\left (M_{ij}+\epsilon_{ij}\right),~~1 \le i \le m_1, 1 \le j \le m_2.
\end{equation}
The major difference between these models is that in the trace-regression model multiple sampling of a particular entry is possible whereas in the Bernoulli model each entry can be sampled at most once. A further difference is that in the trace regression model the number of observations, $n$, is fixed whereas in the  Bernoulli model the number of observations $\hat{n}:=\sum_{ij} B_{ij}$ is random with expectation $E \hat n = n$. Despite these differences, the results on minimax optimal recovery using computationally efficient algorithms for these two models in the literature are very similar and from a `parameter estimation' point of view the models appear to be effectively equivalent (see, e.g., \cite{CaiZhou, Candes_Tao, candes_plan_noise, candes_plan_tight,  Chatterjee, Gross, Keshavan, KloppTrace, Koltchinskiietal,NegahbanWainwright, recht}). \textit{A key insight of the present paper is that for the construction of optimal confidence sets, these models are in fact fundamentally different, at least when the noise variance $\sigma^2$ is unknown.}

When investigating questions that go beyond mere `adaptive estimation' of a high-dimensional parameter, such as about the existence of adaptive confidence sets, one can expect to encounter surprising phenomena -- and various recent results (see e.g. \cite{ Baraud, BullNickl, CaiLow2004, GineNickl10, HoffmannNickl, JuditskyLambert, Low1997, SzaboNickl, NicklVandegeer, RobinsVDV, Szaboetal} and Chapter 8.3 in \cite{Ginenickl}) show that the answers depend on a rather subtle interaction of certain `information geometric' properties of the model -- the material relevant for the present paper is reviewed in Section \ref{GeometryConfidence}. Many of these results reveal limitations by showing that confidence regions that adapt to the whole parameter space do not exist unless one makes specific `signal strength' assumptions. For example, Low \cite{Low1997} and Gin\'e and Nickl \cite{GineNickl10} investigated this question in nonparametric density estimation and Nickl and van de Geer \cite{NicklVandegeer} in the sparse high-dimensional regression model. 

Next to the challenge of adaptation, the construction of confidence sets in the matrix completion setting is difficult mainly due to two reasons. Firstly, the Restricted Isometry Property (RIP) does not hold, requiring a more involved analysis than in a standard trace regression setting such as in \cite{CEGN15}.  Moreover, in most practical applications of matrix completion such as movie recommender systems \cite{BennetLanning, Goldberg} the variance of the errors is not known. Typical constructions of confidence sets in high-dimensions such as $\chi^2$-confidence sets (e.g. \cite{NicklVandegeer, CEGN15}) require explicit knowledge of the variance and are thus not feasible. Particularly in the `Bernoulli model', the problem of unknown variance can be expected to be potentially severe:  for the related standard normal means model (without low rank structure and without missing observations) Baraud \cite{Baraud} has shown  that in the unknown variance case honest confidence sets of shrinking diameter do not exist, even if the true model is low dimensional. Similarly,  in high-dimensional regression Cai and Guo \cite{CaiGuo} prove the impossibility of constructing adaptive confidence sets for the $l_q$-loss, $1 \le q \le 2,$ of adaptive estimators  if the variance is unknown. 

Our main contributions are as follows:  in the trace regression model, even if only an upper bound for the variance of the noise is known, it is shown that practical honest confidence sets exist that have Frobenius-norm diameter that adapts to the unknown rank of $M$. Contrary to this we prove that such confidence regions cannot exist in the Bernoulli model when the noise variance is unknown, and to complement our findings we also prove that in the Bernoulli model with \textit{known} variance, adaptive confidence regions \textit{do} exist. So while recovery algorithms for matrix completion are not sensitive to the choice of model, the task of uncertainty quantification for these algorithms \textit{is}, and crucially depends on the statistician's ability to estimate the noise variance. For the Bernoulli `normal means' model our results imply that the lack of availability of `repeated samples' induces an information-theoretic `barrier' for inference even in the presence of low rank structure.

This paper is organized as follows: in Subsection \ref{sec:not} we formulate the assumptions and collect notation which we use throughout the paper. Then, in Section \ref{GeometryConfidence}, we review and present general results about the existence of honest and adaptive confidence sets in terms of some information-theoretic quantities that determine the complexity of the adaptation problem at hand. Afterwards we review the literature on minimax estimation in matrix completion problems. In Section \ref{sec:trace} we give an explicit construction of honest and adaptive confidence sets in the trace-regression case, adapting a U-statistic approach inspired by Robins and van der Vaart \cite{RobinsVDV} (see also \cite{Ginenickl}, Section 6.4, and \cite{CEGN15}). Finally, we present our results for the Bernoulli model in Section \ref{sec:Bernoulli}. First, we derive an upper bound for the minimax rate of testing a low rank hypothesis and deduce from it the existence of honest and adaptive confidence regions in the known variance case. We then derive a lower bound for this testing rate in the unknown variance case, from which we can deduce that honest and adaptive confidence sets over the whole parameter space cannot exist in general. Sections \ref{sec:Proof1}-\ref{sec:Proofs2} contain the proofs of our results.

\subsection{Notation \& assumptions}\label{sec:not}
By construction, in the Bernoulli model (\ref{Bernoulli_model}) the expected number  of observations, $n$, is  smaller than the total number of matrix entries, i.e. $n\leq m_1m_2 $. To provide a meaningful comparison we will assume throughout that $n \le m_1 m_2$ also holds in the trace regression model (\ref{trace_model}).
In many applications of matrix completion, such as recommender systems (e.g. \cite{BennetLanning, Goldberg}) or sensor localization (e.g. \cite{BiswasLiangWangYe, Singer}) the noise is bounded but not necessarily identically distributed. This is the assumption  which we adopt in the present paper. More precisely, we assume that the $\epsilon_{{\iota}}$ are independent random variables that are homoscedastic, have zero mean and are bounded:
\begin{Assumption}\label{noise_boundedTR}
	In the models (\ref{trace_model}) and (\ref{Bernoulli_model}) with index ${\iota}=i$ and ${\iota}=(i,j)$, respectively, we assume 	$\bE(\epsilon_{\iota})=0$, $\bE(\epsilon_{{\iota}}^{2})=\sigma^{2}$, $\epsilon_\iota \independent \epsilon_\eta$ for $\iota \neq \eta$ and that there exists a positive constant  $U>0$ such that almost surely
	\begin{equation*} 
	\underset{{\iota}}{\max}\left \vert\epsilon_{{\iota}}\right \vert\leq U.
	\end{equation*}
	
\end{Assumption} \noindent
We denote by $M=(M_{ij})\in \mathbb{R}^{m_{1}\times m_{2}}$ the unknown matrix of interest and define \begin{align*}  m& =\min(m_1,m_2), \\ d& =m_1+m_2. \end{align*} For any $l\in \mathbb{N}$ we set  $[l]=\{1,\dots,l \}$. 
Let $A,B$ be matrices in $\mathbb{R}^{m_{1}\times m_{2}}$.
We define the {matrix scalar product} as
$\langle A,B\rangle :=\mathrm{tr}(A^{T}B)$.
The trace norm of the matrix $A$ is defined as  $\Vert A\Vert_{*}:=\sum \sigma_j(A)$, the operator norm as
$\Vert A\Vert:=\sigma_1(A)$ and the Frobenius norm as $\|A \|_F^2:=\sum_i \sigma_i^2 = \sum_{i,j} A_{ij}^2$ 
where  $(\sigma_j(A))$ are the singular values of $A$ arranged in decreasing order. Finally
$\left\Vert A\right\Vert_{\infty}={\max}_{i,j}|A_{ij}|$  denotes the largest absolute value of any entry of $A$. Given a semi-metric $\mathcal D$ we define the diameter of a set $S$ by
$$| S |_{\mathcal D}:=\sup \{ \mathcal D(x,y): ~~x,y \in S \}.$$
Furthermore, for $k \in \mathbb{N}_0$ we define the parameter space of rank $k$ matrices with entries bounded by $a$ in absolute value as 
\begin{equation} \label{def_parameter_space}\mathcal{A}(a,k):=\{ A \in \mathbb{R}^{m_1 \times m_2}: \| A \|_\infty \leq a ~~\text{and}~~\rank(A) \leq k \}. \end{equation}
Finally, for a subset $\Sigma \subset (0,U]$ we define $$\mathcal A(a,k) \otimes \Sigma:= \{ (A, \sigma): ~A \in\mathcal A(a,k),~ \sigma \in \Sigma \}.$$
As usual, for sequences $a_n$ and $b_n$ we say $a_{n} \lesssim b_{n}$ if there exists a constant $C$ independent of $n$ such that $a_n \leq C \cdot b_n$ for all $n$.  We write $\mathbb P_{M,\sigma}$ (and $\mathbb E_{M,\sigma}$ for the corresponding expectation) for the distribution of the observations in the models  (\ref{trace_model}) or (\ref{Bernoulli_model}), respectively.


\section{Minimax theory for adaptive confidence sets} \label{GeometryConfidence}
In this section we present results about existence of honest and adaptive confidence sets in a general minimax framework. To this end, let $Y=Y^n \thicksim \mathbb{P}_f^n$ on some measure space $(\Omega_n , \mathcal{B}), n \in \mathbb{N},$ where $f$ is contained in some parameter space $\mathcal{A}$, endowed with a semi-metric $\mathcal{D}$. Let $r_n$ denote the minimax rate of estimation over $\mathcal A$, i.e.
$$ \inf_{\tilde{f}_n: \Omega_n \rightarrow \mathcal{A}} \sup_{f \in \mathcal{A}} \mathbb{E}_f \mathcal{D}(\tilde{f}, f) \asymp r_n(\mathcal{A}).$$
We consider an `adaptation hypothesis' $\mathcal{A}_0 \subset \mathcal{A}$ characterised by the fact that the minimax rate of estimation in $\mathcal A_0$ is of asymptotically smaller order than in $\mathcal A$: $r_n(\mathcal{A}_0) = o(r_n(\mathcal{A}))$ as $n \rightarrow \infty$. In our matrix inference setting we will choose for $\mathcal{D}$ the distance induced by $\|\cdot\|_F$, for $\mathcal A_0, \mathcal A$ the parameter spaces $\mathcal A(a,k_0)\otimes \Sigma,~ \mathcal A(a,k) \otimes \Sigma$ from above, 
$k_0=o(k)$ as $\min(n,m) \to \infty$, and data $(Y_i,X_i)$ or $(Y_{ij}, B_{ij})$ arising from equation (\ref{trace_model}) or (\ref{Bernoulli_model}), respectively.

\begin{Definition}[Honest and adaptive  confidence sets] 
	Let $\alpha, \alpha'>0$ be given. A set \\ $C_n=C_n(Y, \alpha) \subset \mathcal{A}$ is a honest confidence set at level $\alpha$ for the model $\mathcal A$ if
	\begin{equation} \label{honesty} \liminf_{n}\inf_{f \in \mathcal A} \mathbb P^n_{f}(f \in C_n) \geq 1-\alpha. \end{equation}
	Furthermore, we say that $C_n$ is adaptive for the sub-model  $\mathcal{A}_0$ at level $\alpha'$ if there exists a constant $K=K(\alpha, \alpha') > 0$ such that
	\begin{equation} \label{adaptivity} \sup_{f\in \mathcal{A}_0} \mathbb{P}^n_{f} \left (|C_n|_{\mathcal{D}} > K r_n(\mathcal{A}_0) \right ) \leq \alpha'\end{equation}
	while still retaining
	\begin{equation} \label{adaptivity1} \sup_{f\in \mathcal{A}} \mathbb{P}^n_{f} \left (|C_n|_{\mathcal{D}} > K r_n(\mathcal{A}) \right ) \leq \alpha'.\end{equation}
\end{Definition}
We next introduce certain composite testing problems.
\begin{Definition}[Minimax rate of testing \& uniformly consistent tests] 
	Consider the testing problem
	\begin{equation} \label{SeparationTest} H_0:f \in \mathcal{A}_0 \quad \text{against} \quad H_1 : f \in \mathcal A, ~\mathcal{D}(f,\mathcal{A}_0) \geq \rho_n \end{equation}
	where $(\rho_n:~  n \in \mathbb{N})$ is a sequence of non-negative numbers. We say that $\rho_n$ is the minimax rate of testing for \eqref{SeparationTest} if
	\begin{itemize}
		\item [(i)] $\forall \beta > 0 ~\exists$ a constant $L=L(\beta) > 0$ and a test $\Psi_n=\Psi_n(\beta),~\Psi_n: \Omega_n \rightarrow \{0,1\}$ such that		
		\begin{equation} \label{TestingError} \sup_{f \in \mathcal A_0} \mathbb{E}_{f} [\Psi_n] + \sup_{f \in \mathcal A,~ \mathcal{D}(f, \mathcal A_0) \geq ~L\rho_n } \mathbb{E}_{f} \left [ 1-\Psi_n\right ] \leq \beta. \end{equation}
		We say that such a test $\Psi_n$ is $\beta$-uniformly consistent.
		\item [(ii)] For some $\beta_0 >0$ and any sequence $\rho^*_n=o(\rho_n)$ we have
		\begin{equation} \label{testbeta0}\liminf_{n \rightarrow \infty} \inf_{\Psi_n: \Omega_n \rightarrow \{0,1\}} \left [	\sup_{f \in \mathcal A_0} \mathbb{E}_{f} [\Psi_n] + \sup_{f \in \mathcal A, ~ \mathcal{D}(f,\mathcal{A}_0) \geq \rho_n^* } \mathbb{E}_{f} \left [ 1-\Psi_n \right ] \right ] \geq \beta_0 > 0.\end{equation}
	\end{itemize}
\end{Definition}
\begin{Theorem} \label{Thm_NonExistence}
	Let $\rho_n$ be the minimax rate of testing for the testing problem \eqref{SeparationTest} and suppose that $ \beta_0 > 0$ is as in \eqref{testbeta0}. Suppose that 
	$$r_n(\mathcal{A}_0) = o(\rho_n).$$ Then a honest and adaptive confidence set $C_n$ that satisfies \eqref{honesty}-\eqref{adaptivity1} 
	for any $\alpha, \alpha' > 0$ such that $0 < 2\alpha + \alpha' < \beta_0$ 
	does not exist. In fact if $3\alpha < \beta_0$, then for any honest confidence set $C_n$ that satisfies $\eqref{honesty}$ we have that
	\begin{equation} \label{lbde}
	\sup_{f \in \mathcal{A}_0 } \mathbb{E}_f |C_n|_{\mathcal{D}} \geq c \rho_n.
	\end{equation}
	for a constant $c=c(\alpha) > 0$.
\end{Theorem}
The first claim of this theorem is Proposition 8.3.6 in \cite{Ginenickl}. The lower bound (\ref{lbde}) also follows from that proof, arguing as in the proof of Theorem 4 in \cite{CarpentierNickl}. 

\smallskip

A converse of Theorem \ref{Thm_NonExistence} also exists, as can be extracted from Proposition 8.3.7 in \cite{Ginenickl} and an observation in Carpentier (see \cite{Carpentier}, proof of Theorem 3.5 in Section 6).
For this we need the notion of an \textit{oracle}-estimator. \begin{Definition}[Oracle estimator] Let $\beta > 0$ be given. We say that an estimator $\hat f$ satisfies an oracle inequality at level $\beta$ if there exists a constant $C$ such that for all $f \in \mathcal A$ we have with $\mathbb{P}_{f}^n$-probability at least $1-\beta$,
	\begin{equation} \label{Oracleinequality}
	\mathcal{D}(\hat f, f) \leq C \inf_{\tilde{\mathcal{A}} \in \{ \mathcal{A}, \mathcal{A}_0 \} }\left ( \mathcal{D}(f,\tilde{\mathcal{A}}) + r_n(\tilde{\mathcal{A}}) \right ).
	\end{equation}
\end{Definition}
This is a typical property of adaptive estimators, and is for example in the trace-regression setting fulfilled by the soft-thresholding estimator proposed by Koltchinskii et.al. \cite{Koltchinskiietal}. The following theorem proves that if the minimax rate of testing is no larger than the minimax rate of estimation in the adaptation hypothesis, then honest adaptive confidence sets do exist. The proof is constructive and yields a confidence set of non-asymptotic coverage at least $1-\alpha$. 
\begin{Theorem} \label{Thm_Existence} Let $\alpha, \alpha' > 0$ be given. 
	Let $\rho_n$ be the minimax rate of testing for the problem \eqref{SeparationTest} such that a $\min(\alpha/2, \alpha')$-uniformly consistent test exists. Assume that $\rho_n \leq C'r_n(\mathcal{A}_0)$ for some constant $C'=C(\alpha, \alpha')>0$. Moreover, assume that an oracle estimator $\hat f$ at level $\alpha/2$ fulfilling \eqref{Oracleinequality} exists. Then there exists a confidence set $C_n$ that adapts to the sub-model $\mathcal{A}_0$ at level $\alpha'$ satisfying 	\eqref{adaptivity}, \eqref{adaptivity1} and that is honest at level $\alpha$, i.e., $$ \sup_{f \in \mathcal{A}} \mathbb{P}^n_{f} \left ( f \notin C_n \right )  \leq \alpha. $$
	
\end{Theorem}

\section{Minimax matrix completion}\label{matrixres}
Noisy matrix completion has been extensively studied in several papers starting from Candes and Plan \cite{candes_plan_noise}, see e.g. \cite{candes_plan_tight, Keshavan, Koltchinskiietal, NegahbanWainwright, KloppTrace, Chatterjee,CaiZhou, KloppThresh, recht}).
Optimal rates have been achieved under various sets of assumptions. For instance the construction of the estimator (and the resulting upper bound) in \cite{NegahbanWainwright} requires knowledge of the `spikiness' ratio of the unknown matrix and leads to sub-optimal rates in the case of sparse matrices. The bounds due to Keshavan et. al. \cite{Keshavan} are also only optimal for certain classes of matrices, namely almost square matrices that have a condition number bounded by a constant and fulfil the incoherence condition introduced by \cite{CandesRecht}. Optimal convergence rates for the classes $\mathcal A(a,k)$ of matrices under consideration in the present paper have been obtained by Koltchinskii. et. al. \cite{Koltchinskiietal} and Klopp \cite{KloppTrace} for the trace-regression model and by Klopp \cite{KloppThresh} for the Bernoulli model.
For example, in the trace-regression setting, Klopp \cite{KloppTrace} shows that a constrained Matrix Lasso estimator $\hat M := \hat M(a,\sigma)$ satisfies with $\mathbb{P}_{M_0, \sigma}$-probability at least $1-2/d$ 
\begin{equation} \label{upper_bound_matrix_lasso}
\dfrac{\Vert \hat M-M_0\Vert_F^{2}}{m_1m_2}\leq C \dfrac{
	kd \log(d)}{n} \quad \text{and} \quad \Vert M_0 - \hat M\Vert_{\infty}\leq 2a
\end{equation}
as long as $m \log(d) \leq n \leq d^2 \log(d)$ and where $C=C(\sigma, a) > 0$.
Similarly, in the Bernoulli model with noise bounded by $U$ it has been shown in Klopp \cite{KloppThresh} that an iterative soft thresholding estimator $\hat M := \hat M(a,\sigma)$ satisfies with $\mathbb{P}_{M_0, \sigma}$-probability at least $1-8/d$
\begin{equation} \label{upper_bound_thresholding}
\dfrac{\Vert \hat M-M_0\Vert_{F}^{2}}{m_1m_2}\leq C\frac{
	kd}{n} \quad \text{and} \quad \Vert M_0 - \hat M\Vert_{\infty}\leq 2a
\end{equation} 
for $n \geq m \log(d)$ and 
for a constant $C=C(\sigma, a, U)> 0$. 
Matching lower bounds have also been shown by Koltchinskii. et a. \cite{Koltchinskiietal} and Klopp \cite{KloppThresh}. In the trace-regression model with Gaussian noise we have for constants $\beta \in (0,1)$ and $c=c(\sigma, a) > 0$ that
\begin{equation*} 
\inf_{\hat M} \sup_{M_0 \in \mathcal{A}(a,k)} \mathbb{P}_{M_0, \sigma} \left (   \dfrac{\|\hat M- M_0\|^2_F}{m_1m_2} > c\frac{
	kd }{n} \right ) \geq \beta.
\end{equation*}
A similar lower bound can be obtained in the Bernoulli setting (see Klopp \cite{KloppThresh}).
These lower and upper bounds imply that for the Frobenius loss and the parameter space $\mathcal{A}(a,k)$ the minimax rate $r_{n,m}(\mathcal{A}(a,k))$  is (at most up to a log-factor) of order 
\begin{equation} \label{mrisk}  
\sqrt{m_1 m_2 kd/n}.
\end{equation}
\section{Trace Regression Model \label{sec:trace} }

We first consider the trace regression model. 
For the sake of precision we sometimes write $M_0$ for the `true parameter' $M$ that has generated the equation (\ref{trace_model}). \\
For notational simplicity we assume that $n$ is even. Then we can split our observations in two independent sub-samples of equal size $n/2$. In what follows all probabilistic statements are under the distribution $\pp$ (with corresponding expectation written $\mathbb E$) of the first sub-sample $(Y^{tr}_i,X_i)_{i \leq n/2}$ of size $n/2 \in \mathbb N$, conditional on the second sub-sample $(Y^{tr}_i,X_i)_{i > n/2}$, i.e.~we have $\mathbb P(.) = \mathbb P_{M_0, \sigma}(\,\cdot\,| (Y^{tr}_i,X_i)_{i > n/2})$.

\subsection{A non-asymptotic confidence set in the trace regression model with known variance of the errors.}
In this case we can adapt the construction of \cite{CEGN15}: we first unbiasedly estimate the risk $\| \hat M - M_0 \|_F^2/(m_1m_2)$ of a minimax optimal estimator $\hat M$ computed from an independent sample (e.g., via sample splitting) by a natural $\chi^2$-statistic (see \eqref{residual}). The construction of an unbiased estimate requires knowledge of $\sigma^2$, but when available this estimate, enlarged by natural quantile constants, serves as a good proxy for the diameter of the confidence set $C_n$ centred at $\hat M$.

More precisely, using only the second sub-sample $(Y^{tr}_i,X_i)_{i > n/2}$ we compute the matrix lasso estimator from Klopp \cite{KloppTrace} which achieves the bound \eqref{upper_bound_matrix_lasso} with probability at least $1-2/d$.
Then, we freeze $\hat M$ and the second sub-sample.
We define the following residual sum of squares statistic:
\begin{equation}\label{residual}
\hat R_n = \frac{2}{n} \sum_{i \leq n/2}(Y^{tr}_i-\langle X_i, \hat M\rangle)^2 - \sigma^2.
\end{equation}
Given $\alpha>0$, let $\xi_{\alpha, \sigma,U}=\sqrt{2}\sigma U\log(\alpha)$,
$z_\alpha=\log(3/\alpha)$ and, for a $z > 0$, a fixed constant to be chosen, define the confidence set
\begin{equation} \label{RSSconf_trace}
C_n = \left\{A \in \mathbb R^{m_1\times m_2}: \dfrac{\|A-\hat M\|_F^2}{ m_1m_2} \le  2 \left(\hat R_n + z \frac{d}{n} + \frac{\bar z+\xi_{\alpha, \sigma,U}}{\sqrt n}\right)  \right\},
\end{equation}
where $$\bar z^2= \bar z^2(\alpha,d,n, \sigma, z) = z_{\alpha}\sigma^2 \max\left (\dfrac{3\|A-\hat M\|^2_F}{m_1m_2}, 4zd/n\right ).$$ 
It is not difficult to see (using that $x^2 \lesssim y+x/\sqrt n$ implies $x^2 \lesssim y +1/n$) that
\begin{equation} \label{diam_trace}
\ee_{M_0, \sigma} \left [\frac{| C_n|^2_F}{m_1m_2} \bigg | \hat{M} \right ]\lesssim \frac{\|\hat M -M_{0}\|_F^2}{m_1m_2} + \frac{zd + \sigma^{2}z_{\alpha/3}}{n} + \frac{\xi_{\alpha, \sigma,U}}{\sqrt n}.
\end{equation}
Markov's inequality, \eqref{diam_trace} and that $\hat{M}$  is minimax optimal (up to a log-factor) with $\mathbb{P}_{M_0, \sigma}$-probability of at least $1-2/d$ as long as $m \log(d) \leq n \leq d^2  \log(d)$ imply that $C_n$ has an adaptive and up to a log-factor  minimax optimal squared diameter with probability $1-\alpha'$ for any $\alpha' > 2/d$.
The following theorem shows that $C_n$ is also a honest confidence set:
\begin{Theorem} \label{thm_trace_known_variance}
	Let $\alpha >0$, $ \alpha' > 2/d$ and suppose that $ m \log(d) \leq n \leq d^2  \log(d)$, that Assumption~\ref{noise_boundedTR} is satisfied and that $\sigma > 0$ is known. Let $C_n=C_n(Y,\alpha, \sigma)$ be given by (\ref{RSSconf_trace}) with $z > 0$. Then, for every $n \in \mathbb N$ and every $M_0 \in \mathcal{A}(a,m)$,
	$$ \pp_{M_0, \sigma} \left (M_{0} \in C_n \right ) \ge 1-\frac{2\alpha}{3} -  2 e^{-zd/(11a^{2})}.$$  
	Hence, for any $1 \leq k_0 < k \leq m$, $C_n$ is a honest and (up to a  log-factor) adaptive confidence set at the level $\alpha$  for the model $\mathcal{A}(a,k) \otimes \{ \sigma\}$ and adapts to the sub-model $\mathcal{A}(a,k_0) \otimes \{\sigma \}$ at level $\alpha'$.
\end{Theorem}
The proof of Theorem \ref{thm_trace_known_variance} follows the lines of the proof of Theorem 2 in \cite{CEGN15} and we omit it here as the unknown variance results considered in the next section straightforwardly imply the known variance results.

\subsection{A non-asymptotic confidence set in the trace regression model with unknown error variance. }

In this subsection we assume,  that the precise knowledge of the noise variance $\sigma$ is \textit{not} available, although the quantities $a,U$ are available to the statistician (i.e.~upper bounds on the matrix entries and on the noise). More precisely we assume that $\sigma$ belongs to a known set $\Sigma \subset (0,U]$. In applications of matrix completion this is usually a realistic assumption since the entries of $M_0$ are bounded: For example in a movie recommender system (e.g. \cite{BennetLanning, Goldberg}) the entries of the observations $Y$ and consequently $M_0$ and $\epsilon_i$ are bounded from above by the best possible rating and below from the worst possible rating.\\
As the variance is now assumed to be unknown the construction from \eqref{RSSconf_trace} is not feasible anymore since we can not compute the test statistic \eqref{residual}. Instead we use a U-statistic approach: As in the previous section, we use the second half of the sample,  $(Y_i^{tr}, X_i)_{n/2 < i\leq n }$, for constructing a minimax optimal estimator $\hat M$ of $M$ that fulfills $\| \hat M \|_\infty \leq a$. We use again the matrix lasso estimator from Klopp \cite{KloppTrace} (with $\sigma$ replaced by its upper bound ${U}$) which achieves \eqref{upper_bound_matrix_lasso} with probability at least $1-2/d$. 
In order to construct the confidence set, we will be interested in all pairs of observations $(Y_l^{tr}, X_l)$ and $( Y_s^{tr}, X_s)$ in the first sub-sample with $1 \leq l<s \leq n/2$ such that $X_l=X_s$ (that is, independent measurements of the same matrix entry). For each $(i,j)\in [m_1]\times [m_2]$, let $\mathcal S_{(i,j)} = \{k \in \{1, \dots, n/2\}  : X_k = e_i(m_1)e^{T}_j(m_2)\} =: \{a_1 <...<a_{p_{(i,j)}}\}$ where $p_{(i,j)}$ is the number of times that we observe the entry $(i,j)$. For all indices $(i,j)$ such that $S_{(i,j)}\not=\emptyset$, we form
the $\lfloor p_{(i,j)}/2 \rfloor$ couples  $(X_{a_{1}},X_{a_{2}}),(X_{a_{3}},X_{a_{4}}),\dots \text{etc.}$
We denote by $\mathcal{N}$ the set of all these pairs and let $\vert \mathcal{N}\vert =N $ be their number. Re-ordering, we can write $(\tilde X_k, Z_k, Z_k')_{k \leq N}$ where $\tilde X_k = X_l=X_s$ for some couple $(X_l,X_s)\in \mathcal{N}$ and $Z_k = Y_l^{tr}$ and $Z_k' = Y_s^{tr}$. That is, using two different samples of the same entry $\tilde X_k = X_l=X_s$ we form the observation triples $(\tilde X_k, Z_k, Z_k')$.
We use $(\tilde X_k, Z_k, Z_k')_{k \leq N}$ to construct a U-Statistic to estimate the squared Frobenius loss. Contrary to the construction in \eqref{residual} this does not require  knowledge of the variance of the errors.
We define:
\begin{equation}
\hat R_N:= \frac{1}{N} \sum_{k=1}^N (Z_k- \langle \hat{M}, \tilde X_k \rangle)(Z_k'- \langle \hat{M}, \tilde X_k \rangle), 
\end{equation}
and we set $\hat{R}_N = 0$ if $N=0$. Note that
\begin{equation} \label{UStatisticExp}
\mathbb{E}_{M_0, \sigma}   \bigg [\hat{R}_N \bigg | \hat{M}, N \geq 1 
\bigg ] = \frac{\| \hat{M} - M_0 \|_F^2}{m_1 m_2}.
\end{equation}
We define the confidence set
\begin{equation} \label{Ustatci}
C_n:= \left \{ A \in \mathcal{A}(a,m):~~ \frac{\| A - \hat M \|_F^2}{m_1 m_2} \leq \hat{R}_N + z_{\alpha,N}\right \}
\end{equation}
where the random quantile constant $z_{\alpha, N}$ is defined as $$z_{\alpha, N}:=\frac{U^2+4a^2}{\sqrt{N  \alpha}}\quad\text{if}\quad N \neq 0\quad\text{and}\quad z_{\alpha, N} =  4a^2 \quad \text{if~ } N=0.$$
The quantity $N$ is random but we can bound it from below with high probability by 
$n^2 / (64m_1 m_2)$ as proven in the following lemma.
\begin{Lemma} \label{Kombinatorikinequality}
	For $n \leq m_1 m_2$ we have with probability at least $1-\exp \left ( -n^2/(372 m_1 m_2 )\right )$ that:
	\begin{equation*}
	N \geq \frac{n^2}{64  m_1 m_2}.
	\end{equation*}
\end{Lemma}
Markov's inequality, \eqref{UStatisticExp},  Lemma \ref{Kombinatorikinequality} and that $\hat{M}$ achieves the nearly optimal rate \eqref{upper_bound_matrix_lasso} with $\mathbb{P}_{M_0, \sigma}$-probability of at least $1-2/d$ imply for any $k \le m$, any $M_0 \in \mathcal{A}(a,k)$,  any $\sigma \leq U$, any $\alpha' > 2/d + \exp(-n^2/(372m_1m_2))$ and a large enough constant $C=C(\alpha, \alpha',\sigma, a, U) > 0$ that \begin{equation}
\mathbb{P}_{M_0, \sigma} \left (\frac{|C_n|_F^2}{m_1m_2} > C  \frac{k d\log(d)}{n} \right ) \leq \alpha'.
\end{equation}
Since $k$ is arbitrary this implies that $C_n$ is a confidence set whose $\|\cdot\|_F^2$-diameter adapts to the unknown rank of $M_0$ without requiring the knowledge of $\sigma \in \Sigma$. The following theorem implies that $C_n$ is also a honest confidence set.
Note that our result is non-asymptotic and holds for any triple $(n,m_1,m_2) \in \mathbb{N}^3$ as long as $m\log d \leq n \leq  m_1 m_2$.
\begin{Theorem} \label{traceustatistic} Let $\alpha > 0$ be given, assume $m \log(d) \leq n \leq  m_1 m_2$ and that Assumption~\ref{noise_boundedTR} is fulfilled. Let $C_n=C_n(Y,\alpha)$ as in (\ref{Ustatci}). Then $C_n$ satisfies for any $M_0 \in \mathcal{A}(a,m)$ and any $\sigma \in \Sigma$
	\begin{align*}
	\mathbb{P}_{M_0, \sigma} \left (  M_0 \in C_n\right )  \geq 1- \alpha.
	\end{align*}
	Hence, for any $\alpha' > 2/d +  \exp (-n^2/(372m_1m_2))$ and any $1 \leq k_0 < k \leq m$, $C_n$ is a honest confidence set at level $\alpha$  for the model $\mathcal{A}(a,k) \otimes \Sigma$ that adapts (up to a log-factor) to the rank $k_0$ of any sub-model $\mathcal{A}(a,k_0) \otimes \Sigma$ at level $\alpha'$.
	
\end{Theorem} 


\section{Bernoulli Model \label{sec:Bernoulli}}
In this section we consider the Bernoulli model \eqref{Bernoulli_model}. As before we let $\mathbb P_{M,\sigma}$ (and $\mathbb E_{M,\sigma}$ for the corresponding expectation) denote the distribution of the data when the parameters are $M$ and $\sigma$, and we sometimes write $M_0$ for the `true' parameter $M$ for the sake of precision.

\subsection{A non-asymptotic confidence set in the Bernoulli model with known variance of the errors.}
Here we assume again that $\sigma > 0$ is known. In case of the Bernoulli model we are not able to obtain two independent samples and cannot use the risk estimation approaches from the trace-regression setting. Instead we use the duality between testing and honest and adaptive confidence sets laid out in Section \ref{GeometryConfidence}. We first determine an upper bound for the minimax rate $\rho=\rho_{n,m}$ of testing the low rank hypothesis
\begin{equation} \label{CompositeTestingProblem} H_0: M \in \mathcal{A}(a,k_0)~~\text{against}~~H_1: M \in \mathcal{A}(a,k),~~\| M - \mathcal{A}(a,k_0) \|_F^2 \geq \rho^2, \end{equation}
and then apply Theorem \ref{Thm_Existence}. As test statistic, we propose an infimum-test which has previously been used by Bull and Nickl \cite{BullNickl} and Nickl and van de Geer \cite{NicklVandegeer} in density estimation and high-dimensional regression, respectively (see also Section 6.2.4. in \cite{Ginenickl}). Since $\sigma^2 = \mathbb{E}\epsilon_{ij}^2$ is known we can define the statistic
\begin{equation} \label{Infimumteststatistic}
T_n:= \inf_{A \in \mathcal{A}(a,k_0)} \left | \frac{1}{\sqrt{2n}} \sum_{i,j}B_{ij}\left ((Y_{ij}-A_{ij})^2-\sigma^2\right ) \right |  = \inf_{A \in \mathcal{A}(a,k_0)} \left | \frac{1}{\sqrt{2n}} \sum_{i,j} \left ( (Y_{ij}-B_{ij}A_{ij})^2-B_{ij}\sigma^2\right ) \right |
\end{equation}
and choose the quantile constant $u_\alpha$ such that
\begin{equation} \label{Quantileconstant3}
\mathbb{P}_{ \sigma}  \left ( \frac{1}{\sqrt{2n}} \left | \sum_{i,j} B_{ij}(\epsilon_{ij}^2-\mathbb{E}\epsilon_{ij}^2) \right | > u_\alpha \right ) \leq \alpha/3.
\end{equation}
For example, using Markov's inequality, we obtain
\begin{align} 
\mathbb{P}_{ \sigma} \left ( \frac{1}{\sqrt{2n} }\left | \sum_{i,j} B_{ij}(\epsilon^{2}_{ij}-\sigma^2)\right | > u_\alpha\right ) & \leq \frac{1}{2n u^2_\alpha}\sum_{i,j} \Var_\sigma \left ( B_{ij}(\epsilon_{ij}^2-\sigma^2)\right )\notag  \leq \frac{\sigma^2 (U^2-\sigma^2)}{2u^2_\alpha} \notag
\end{align}
so $		u_\alpha = \sigma \sqrt{\big (3(U^2-\sigma^2) \big )/ (2\alpha)} \notag	$
is an admissible choice. 
\begin{Theorem} \label{Thm_CompositeTesting} Let $\alpha \geq 12 \exp(-100d)$ be given. Consider the Bernoulli model \eqref{Bernoulli_model} and the two parameter spaces $\mathcal{A}(a,k)$ and $\mathcal{A}(a,k_0)$, $1 \leq k_0 < k \leq m$. Furthermore assume that Assumption \ref{noise_boundedTR} is fulfilled, that $\sigma > 0$ is known, that $n \geq m \log (d)$ and consider the testing problem \eqref{CompositeTestingProblem}. Suppose $$
	\rho^2 \geq C \frac{ m_1 m_2 k_0d}{n}
	\asymp r^2_{n,m}(\mathcal{A}(a,k_0))$$ 
	where $C=C(\alpha, a, U, \sigma) > 0$ is a constant. Then the test
	$
	\Psi_n:=\mathbf{1}_{  \{ T_n > u_\alpha \}}
	$
	where $u_\alpha$ is the quantile constant in \eqref{Quantileconstant3} and $T_n$ is as in $\eqref{Infimumteststatistic}$ fulfills
	$$  \sup_{M \in \mathcal{A}(a,k_0)} \mathbb{E}_{M, \sigma} [\Psi_n ]+
	\sup_{M \in \mathcal{A}(a,k),~ \| M - \mathcal{A}(a,k_0)\|_F^2 \geq \rho^2}  \mathbb{E}_{M, \sigma} [1-\Psi_n] 
	\leq \alpha.$$
\end{Theorem}
Now in order to apply Theorem \ref{Thm_Existence} we use the soft-thresholding estimator proposed by Koltchinskii et. al. \cite{Koltchinskiietal} which satisfies the oracle inequality \eqref{Oracleinequality} up to a log-factor in the trace regression model. That this holds in the Bernoulli-model as well with $\mathbb{P}_{M_0, \sigma}$-probability of at least $1-1/d$ can be proven in a similar way and we sketch this in Proposition \ref{OracleBernoulliThm}, removing the $\log$-factor by using stronger bounds on the spectral norm of the stochastic term $(B_{ij}\epsilon_{ij})_{i,j}$. 

This and Theorem \ref{Thm_CompositeTesting} imply, using Theorem \ref{Thm_Existence}, that there exist honest and adaptive confidence sets in the Bernoulli model if the variance of the errors is known.
\begin{Corollary} Let $\alpha \geq 2/d$ and $\alpha' \geq 12 \exp(-100d)$ be given. Suppose that $\sigma > 0$ is known, that Assumption \ref{noise_boundedTR} is fulfilled and that $n \geq m \log(d)$. Then, for any $1 \leq k_0 < k \leq m$, there exists a honest  confidence set $C_n$ at the level $\alpha$ for the model $\mathcal{A}(a,k) \otimes \{ \sigma \}$, i.e., for any $M_0 \in \mathcal A(a, k)$, $$\mathbb{P}_{M_0, \sigma} \left ( M_0 \in C_n \right ) \geq 1-\alpha,$$ and $C_n$ adapts to the sub-model $\mathcal{A}(a,k_0) \otimes \{ \sigma \}$ at level $\alpha'$.

\end{Corollary}
\subsection{The case of the Bernoulli model with unknown error variance.}
In this subsection we assume again, as in Subsection 5.2, that the precise knowledge of the error variance $\sigma$ is \textit{not} available. Whereas in this case for the trace-regression model the construction of honest and adaptive confidence set was seen to be possible, we will  now show that this is not the case for the Bernoulli model. We use again the duality between testing and confidence sets, this time applying Theorem \ref{Thm_NonExistence}. The next theorem gives a lower bound for the minimax rate of testing for the composite null hypothesis $H_0: M \in \mathcal{A}(a,k_0)$ of $M$ having rank at most $k_0$ against a rank-$k$ alternative. To simplify the exposition we will consider only square matrices (but see the remark below) and also an asymptotic `high-dimensional' framework where $\min (n,m) \to \infty$ and $k_0 = o(k)$. We formally allow for $k_0=0$, thus including the `signal detection problem' when $H_0: M=0, \sigma^2=1$. 

\begin{Theorem}\label{thm:main}
	Suppose that Assumption~\ref{noise_boundedTR} is satisfied for some $U \geq 2$  and assume $m=m_1=m_2$. Furthermore, 
	let $k=k_{n,m} \to \infty $ be such that $0 < k \leq m^{1/3}$ and $k^{1/4}\sqrt{m/n} < \min(1, a)/2$. For $0 \leq k_0<k$ satisfying $k_0=o(k)$ and a sequence $\rho=\rho_{n,m} \in (0,1/2)$ consider the testing problem
	\begin{equation} \label{test1}
	H_0 : M \in \mathcal{A}(a,k_0),~ \sigma^2 = 1 \quad \text{vs} \quad H_1 : M \in \mathcal{A}(a,k), ~\|M-\mathcal{A}(a,k_0)\|_F^2 \geq m^2\rho^2, ~\sigma^2 = 1 - 4\rho^2. \end{equation}
	If as $\min (n,m) \to \infty$, 
	\begin{equation} \label{SeparationRateUnknown}
	\rho^2 =o \Big(\frac{\sqrt{k}m}{n} \Big),
	\end{equation}
	then for any test $\Psi$ we have that
	\begin{equation}
	\liminf_{\min(n,m) \rightarrow \infty} \left [\sup_{M \in \mathcal{A}(a,k_0)} \mathbb{E}_{M, 1} [\Psi] + \sup_{M \in \mathcal{A}(a,k), ~\|M-\mathcal{A}(a,k_0)\|_F^2 \geq m^2\rho^2} \mathbb{E}_{M, \sqrt{1-4\rho^2}} [1-\Psi] \right ] \geq 1.
	\end{equation}
	In particular,  if $\Sigma \subset (0,U]$ contains the interval $[\sqrt{1-4\tau}, 1 ]$ where $\tau = \limsup_{n,m} k^{1/4} \sqrt {m/n},$ 
	then \eqref{testbeta0} holds for the choices $\mathcal{A}_0=\mathcal{A}(a,k_0) \otimes \Sigma, \mathcal{A}=\mathcal{A}(a,k) \otimes \Sigma$ and $\beta_0=1, \rho^*=\rho$. 
\end{Theorem}
Using Theorem \ref{Thm_NonExistence} this implies the non-existence of honest and adaptive confidence sets in the model (\ref{Bernoulli_model}) if the variance of the errors is unknown and $k_0 = o(\sqrt{k})$. In particular adaptation to a constant rank $k_0$, $k_0=O(1)$, is never possible if $k \to \infty$ as $\min(m,n) \to \infty$.
\begin{Corollary}
	Assume that the conditions of Theorem \ref{thm:main} are fulfilled and that $k_0=o(\sqrt{k})$. Then 
	for any $\alpha, \alpha' > 0$  satisfying $0 < 2\alpha + \alpha' < 1$ a honest confidence set for the model $\mathcal A(a,k) \otimes \Sigma$ at level $\alpha$ that adapts to the sub-model  $\mathcal A(a,k_0) \otimes \Sigma$ at level $\alpha'$ does not exist. In fact if $\alpha < 1/3$, we have for every honest confidence set $C_n$ for the model $\mathcal{A}(a,k) \otimes \Sigma$ at level $\alpha$ and constant $c=c(a, U, \alpha)$ that
	$$ \sup_{(M_0, \sigma) \in \mathcal A(a,k_0) \otimes \Sigma}\mathbb{E}_{M_0, \sigma}  |C_n|_F^2\geq c \frac{m^3\sqrt{k}}{n}.$$
\end{Corollary}

The above results are formulated for square matrices ($m_1=m_2$) to keep the technicalities in the proof at a reasonable level. One can adapt the proof of Theorem \ref{thm:main} to obtain a lower bound of the order $\rho^2 \gtrsim \sqrt{k m_1 m_2}/n$ which likewise leads to non-existence results for adaptive confidence sets for non-square matrices in relevant asymptotic regimes of $k_0, k, m_1,m_2$.

\section{Conclusions} 
We have investigated confidence sets in two matrix completion models: the Bernoulli model and the trace regression model. In the trace regression model the construction of adaptive confidence sets is possible, even if the variance is unknown. Contrary to this we have shown that the information theoretic structure in the Bernoulli model is different; in this case the construction of adaptive confidence sets is not possible if the variance is  unknown. 

One interpretation is that in practical applications (e.g. recommender systems such as Netflix \cite{BennetLanning}) one should incentivise users to perform multiple ratings, to justify the use of the trace regression model and the proposed U-statistic confidence set.

Our proof only shows that one can not adapt to general low rank hypotheses if the variance is unknown. This covers the key cases where $k_0=1$ or more generally $k_0=o(\sqrt k)$. It remains an interesting open (and difficult) question whether the lower bound $\rho$ in Theorem \ref{thm:main} is tight, but the answer to this question does not affect the main conclusions of our results on the existence of adaptive confidence sets in matrix completion problems.

\section{Proofs \label{sec:Proof1}}

\subsection{Proof of Theorem~\ref{Thm_Existence}}
\begin{proof}
	Let $\Psi_n$ be a test that attains the rate $\rho$ with error probabilities bounded by $\min (\alpha/2, \alpha')$ and let $L=L(\min (\alpha/2, \alpha'))$ be the corresponding constant in \eqref{TestingError}. Let $\hat f$ denote an estimator that satisfies the oracle inequality \eqref{Oracleinequality} with probability of at least $ 1-\alpha/2$. Define a confidence set
	$$C_n:= \{ f \in \mathcal{A}:~ \mathcal{D}(\hat f, f) \leq K\left ( r_n(\mathcal{A}) \Psi_n + r_n(\mathcal{A}_0)(1-\Psi_n)\right )  \}$$
	where $K > 0$ is a constant to be chosen.\\
	We first prove that $C_n$ is adaptive: If $f \in \mathcal{A} \backslash \mathcal{A}_0$ there is nothing to prove, and if $f \in \mathcal{A}_0$ we have
	$$\mathbb{P}_{f}^n \left ( |C_n|_{\mathcal{D}} > K r_n ( \mathcal{A}_0) \right ) = \mathbb{P}_{f}^n (\Psi_n=1) \leq \alpha'.$$
	For coverage we investigate three distinct cases and note that
	\begin{equation} \label{AdaptiveEstimator} \sup_{f \in \tilde{\mathcal{A}}} \mathbb{P}_{f}^n \left (  \mathcal{D}(\hat f,f ) > C r_n (\tilde{\mathcal{A}})\right ) \leq \alpha/2 \end{equation} where $C > 0$ is as in $\eqref{Oracleinequality}$ and where $\tilde{\mathcal{A}} \in \{ \mathcal{A}_0, \mathcal{A} \}$. Hence $\hat f$ is, by the oracle inequality, an adaptive estimator. \\Then for $f \in \mathcal{A}_0$, by \eqref{AdaptiveEstimator}
	$$ \mathbb{P}_{f}^n ( f \notin C_n ) \leq \mathbb{P}_{f}^n \left ( \mathcal{D}(\hat f, f) > Kr_n(\mathcal{A}_0) \right ) \leq \alpha/2 \leq \alpha$$ for $K \geq C$. \\If $f \in \mathcal{A} \backslash \mathcal{A}_0$ and $\mathcal{D}(f, \mathcal{A}_0) \geq L\rho_n$, then for $K \geq C$
	\begin{align*} \mathbb{P}_{f}^n ( f \notin C_n) & = \mathbb{P}_{f}^n (\mathcal{D}(\hat f, f) > K r_n(\mathcal{A}), \Psi_n=1) +  \mathbb{P}_{f}^n (\mathcal{D}(\hat f, f) > K r_n(\mathcal{A}), \Psi_n=0)  \\
	& \leq  \mathbb{P}_{f}^n (\mathcal{D}(\hat f, f) > K r_n(\mathcal{A} )) + \mathbb{P}_{f}^n (\Psi_n=0) \leq \alpha. \end{align*}
	If $f \notin \mathcal{A} \backslash \mathcal{A}_0$ but $\mathcal{D}(f, \mathcal{A}_0) < L\rho_n$, then by the oracle inequality and since $\rho_n \leq C' r_n(\mathcal{A}_0)$  we have with probability at least $1-\alpha/2$ for such $f$ that
	$$\mathcal{D}(\hat f, f) \leq C (\mathcal{D}(f, \mathcal{A}_0)+r_n(\mathcal{A}_0)) \leq CL\rho_n + C r_n(\mathcal{A}_0) \leq C(LC'+1)r_n(\mathcal{A}_0).$$
	Thus we still have
	$$ \mathbb{P}^n_{f} \left ( f \notin C_n \right ) = \mathbb{P}_{f}^n (\mathcal{D}(\hat f, f) > K r_n(\mathcal{A}_0)) \leq \alpha/2 \leq \alpha$$ for $K \geq C(LC'+1)$.
\end{proof}

\subsection{Proof of Theorem~\ref{traceustatistic}}

\begin{proof}
	{Recall that} 
	\begin{equation}\label{eq:esp}
	\mathbb E_{M_0, \sigma}\Big(\hat{R}_N|N, N >0\Big)=\frac{\|\hat M - M_0\|_F^2}{m_1m_2} =:r.
	\end{equation}
	Thus using Markov's inequality we have for $N > 0$ that
	\begin{align} \label{Varianceestimate}
	\mathbb{P}_{M_0, \sigma} \left (  M_0 \notin C_n | N, N > 0\right )& \leq \mathbb{P}_{M_0, \sigma} \left ( |\hat{R}_N - r| > z_{\alpha,N} | N, N>0 \right ) \notag\\
	& \leq \frac{\mathbb \Var_{M_0, \sigma}\Big(\hat{R}_N \big| N, N >0 \Big)}{z_{\alpha,N}^2}.
	\end{align}
	Using equation~\eqref{eq:esp} we compute
	\begin{align*}
	\mathbb \Var_{M_0, \sigma}\Big(\hat{R}_N \big| N, N>0 \Big) 
	&=\frac{1}{N}  \mathbb E_{M_0, \sigma}\Big(\big((Z_k- \langle \hat{M}, \tilde X_k \rangle)(Z_k'- \langle \hat{M}, \tilde X_i \rangle)-r\big)^2  \Big| N, N>0  \Big) \\
	& \leq \frac{1}{N} \left [ \left (\mathbb E \langle M_0-\hat{M},  X_1 \rangle^4 \right ) + 2\sigma^2 r + \sigma^4 \right ] \\
	& = \frac{1}{N} \left [\frac{\| \hat M - M_0 \|_{L^4}^4}{m_1 m_2}  + 2 \sigma^2 r +\sigma^4 \right ] \\
	& \leq \frac{U^4+8 U^2 a^2+16 a^4}{N} = \alpha z_{\alpha, N}^2
	\end{align*}
	since $\| \hat M - M_0 \|_\infty \leq 2a$ and where we define $\|\hat M - M_0 \|_{L^4}^4 := \sum_{i,j} (\hat M_{ij} - {M}_{ij})^4$. Hence ~\eqref{Varianceestimate} implies 
	\begin{align*}
	\mathbb{P}_{M_0, \sigma} \left (  M_0 \notin C_n | N> 0\right )  \leq \alpha.
	\end{align*}
	Moreover, as $\| \hat M - M_0 \|_\infty \leq 2a$ and  $z_{\alpha, 0} = 4a^2$, we have that $\mathbb{P} \left (  M_0 \notin C_n | N=0\right ) =0$.


\end{proof}

\subsection{Proof of Theorem~\ref{Thm_CompositeTesting}}
\begin{proof}
	If $M \in \mathcal{A}(a,k_0)$, then by definition of the infimum and $u_\alpha$ we have
	\begin{align*} \mathbb{E}_{M, \sigma} [\Psi] & = \mathbb{P}_{M, \sigma} \left ( T_n > u_\alpha \right )  \leq  \mathbb{P}_{ \sigma} \left ( \frac{1}{\sqrt{2n}} \left | \sum_{ij}B_{ij}(\epsilon_{ij}^2-\sigma^2) \right | > u_\alpha \right ) \leq \alpha/3.
	\end{align*}
	The case $M \in \mathcal{A}(a,k), \| M - \mathcal{A}(a,k_0) \|_F^2 \geq \rho^2$ requires more elaborate arguments. Let $A^*$ be a minimizer in \eqref{Infimumteststatistic}. Then
	\begin{align}
	\mathbb{E}_{M, \sigma}[1-\Psi] & = \mathbb{P}_{M, \sigma} \left ( T_n < u_\alpha \right ) \notag\\
	& = \mathbb{P}_{ \sigma} \left ( \left | \sum_{ij} B_{ij}[(A^*_{ij}-{M}_{ij})^2-2 \epsilon_{ij}(A_{ij}^*-{M}_{ij})+(\epsilon_{ij}^2-\sigma^2)]  \right | < \sqrt{2n} u_\alpha \right ). \label{zerlegung}
	\end{align}
	For $\rho \geq 8072a \sqrt{k_0d/p} = 8072a \sqrt{m_1 m_2 k_0d/n} $ we can apply Lemma \ref{lem_bernoulli_rip} which yields a weaker version of the Restricted Isometry Property (RIP).  Namely, Lemma \ref{lem_bernoulli_rip} implies that the event $$\Xi:= \left \{ \sum_{i,j} B_{ij}({A}_{ij}-M_{ij})^2 \geq \frac{p}{2} \|A - M \|_F^2 ~~~\forall A \in \mathcal{A}(a,k_0) \right \},~~~ M \in H_1,$$ occurs with probability of at least $1-2 \exp(-100d)$. We can thus bound \eqref{zerlegung} by
	\begin{align}
	& \mathbb{P}_{ \sigma} \left ( \sup_{A \in \mathcal{A}(a,k_0)} \left[2 \left | \sum_{i,j} B_{ij} \epsilon_{ij}(A_{ij}-{M}_{ij})  \right | - \frac{\sum_{i,j}B_{ij}({A}_{ij}-M_{ij})^2}{2}\right] > -\sqrt{n} u_\alpha , \Xi\right )\label{InfTest_CrossTerm} \\ + & \mathbb{P}_{ \sigma} \left ( \left | \sum_{i,j} B_{ij}(\epsilon_{ij}^2-\sigma^2) \right | > \frac{\sum_{i,j}B_{ij}(A^*_{ij}-M_{ij})^2}{2} - \sqrt{n} u_\alpha, \Xi \right ) + 2 \exp(-100d).  \label{InfTest_StochTerm}
	\end{align}
	The stochastic term \eqref{InfTest_StochTerm} can be bounded using $d^2 \geq 3n$ and that $\rho$ is large enough. Indeed, on the event $\Xi$  we have that $$
	\frac{\sum_{i,j}B_{ij}(A_{ij}^*-M_{ij})^2}{2} \geq p \rho^2/4 \geq (1+\sqrt{2})/\sqrt{3} d u_\alpha \geq (1+\sqrt{2}) \sqrt{n} u_\alpha$$ for $\rho \geq 2 \sqrt{u_\alpha d /p}$  which implies together with the definition of $u_\alpha$ in \eqref{Quantileconstant3} that $\eqref{InfTest_StochTerm}$ can be bounded by $\alpha/3 + 2 \exp(-100d)$.
	For the cross term \eqref{InfTest_CrossTerm} we use the two following inequalities which, just as before, hold on the event $\Xi$ $\forall~A \in \mathcal{A}(a,k_0)$ $$\frac{\sum_{i,j}B_{ij}({A}_{ij}-M_{ij})^2}{4} \geq \sqrt{n} u_\alpha ~~\text{and} ~~\frac{\sum_{i,j}B_{ij}({A}_{ij}-M_{ij})^2}{8} \geq \frac{p \|A-M  \|_F^2}{16}.$$ Hence, using also 
	a peeling argument, \eqref{InfTest_CrossTerm} can be bounded by
	\begin{align}
	& \sum_{s \in \mathbb{N}:~ p \rho^2/2 \leq 2^s < \infty} \mathbb{P}_{ \sigma} \left ( \sup_{A \in \mathcal{A}(a,k_0), ~2^{s} \leq p \| A-M \|_F^2 \leq 2^{s+1}} \frac{ \left | \sum_{i,j}B_{ij}\epsilon_{ij}(A_{ij}-{M}_{ij}) \right |}{p\| A-M \|_F^2 }  > \frac{1}{16} \right ) \notag \\
	\leq & \sum_{s \in \mathbb{N}:~ p \rho^2/2 \leq 2^s < \infty} \mathbb{P}_{ \sigma} \left ( \sup_{A \in \mathcal{A}(a,k_0),~ p \| A-M \|_F^2 \leq 2^{s+1}}  \left | \sum_{i,j} B_{ij}\epsilon_{ij}(A_{ij}-{M}_{ij}) \right | > \frac{2^s}{16} \right ) \notag \\
	= & \sum_{s \in \mathbb{N}:~ p \rho^2/2 \leq 2^s < \infty} \mathbb{P}_{ \sigma} \left ( Z(s) > \frac{2^s}{16} \right )
	\label{Peeling3}
	\end{align}
	where we set the corresponding probability to $0$ if the supremum is taken over an empty set and where we define $$Z(s):=\sup_{A \in \mathcal{A}(a,k_0), ~p \| A-M \|_F^2 \leq 2^{s+1}}  \left | \sum_{i,j} B_{ij}\epsilon_{ij}(A_{ij}-{M}_{ij}) \right |.$$
	Lemma \ref{Lemma:Talagrand} (with choices $z=16^2$,  $\xi_{ij}=\epsilon_{ij}$, $t=2^s$ and $q=1$ there ) implies for $\rho \geq 16144U\sqrt{k_0d/p}$ and for $2^s \geq p\rho^2 /2$  that
	$$ \mathbb{P}_\sigma \left ( Z(s) > \frac{2^s}{16} \right ) \leq \exp \left ( \frac{-2^s}{2097152U^2+517120aU }\right )$$
	Hence, \eqref{Peeling3} can be upper bounded by
	\begin{align}
	\sum_{s \in \mathbb{N}:~p \rho^2/2 \leq 2^s < \infty}  
	\exp \left ( \frac{-2^s}{2097152U^2+517120aU }\right )
	\leq  & 2 \exp \left ( -\frac{p \rho^2}{2097152U^2+517120aU } \right ) \\  \leq & 2 \exp (- 100d)  \notag
	\end{align}
	for $\rho \geq  16169U(a \vee U) \sqrt{d/p}$. Consequently \eqref{zerlegung} can be bounded by $\alpha/3 + 4 \exp(-100d) \leq 2\alpha/3$ since $\alpha \geq 12 \exp(-100d)$.
\end{proof}
\subsection{Proof of Theorem~\ref{thm:main}}

\begin{proof}
	\textbf{Step I : Reduction to an easier testing problem between two distributions} \\ 
	Assume without loss of generality 
	that $m$ is divisible by $k$. Suppose \begin{equation} \label{rho_n} \rho=\rho_{n,m}=\frac{vk^{1/4}\sqrt{m}}{\sqrt{n}} \end{equation} where $v=v_{n,m}$ is a sequence such that $v=o(1)$, and assume w.l.o.g.~that $0<v\le 1$. Moreover we denote $u =2\rho$. 
	For 
	$1 \leq i \leq m,~ 1 \leq \kappa \leq k, ~ 1 \leq j \leq m$ let
	\begin{equation*}
	B_{ij} \overset{i.i.d.}{\thicksim} \mathcal B(p)\quad \text{and}\quad U_i^{\kappa}\overset{i.i.d.}{\thicksim} \mathcal R\quad \text{and}\quad V_j \overset{i.i.d.}{\thicksim} \mathcal R,
	\end{equation*}
	where $\mathcal B(p)$ is a Bernoulli distribution of parameter $p = n/m^2$ and $\mathcal R$ is the standard Rademacher distribution $\Pr (V_1 = \pm 1)=1/2$. Let
	$\mathcal P$
	be a uniform random partition of $\{1, \ldots, m\}$ in $k$ groups of size $m/k$, and denote by $K_j$, $K_j \in \{ 1,...,k \}$, the label of  element $j$ of $\mathcal P$.
	Consider the following testing problem:
	\begin{align*}
	H_0' &: M= 0\quad \text{and}\quad \epsilon_{ij} \overset{i.i.d.}{\thicksim} \mathcal R
	\end{align*}
	$$\text{against}$$
	\begin{align} 
	H_1' : M_{ij}=u U^{K_j}_i V_j \label{prior}\\
	and \quad \epsilon_{ij} \thicksim \delta_{\{1-M_{ij} \}} (1+M_{ij})/2 &+ \delta_{ \{ -1 - M_{ij} \} } (1-M_{i,j})/2 \notag
	\end{align}
	Note that the variance of $\epsilon_{ij}$ under $H_0$ is $1$ and the variance of the noise under $H_1$ is $$(1-M_{ij})^2(1+M_{ij})/2 + (-1-M_{ij})^2(1-M_{ij})/2 = (1-M_{ij})(1+M_{ij}) = 1-4\rho^2,$$ so the noise variables are homoscedastic across the $(i,j)$'s and $|\epsilon_{ij}| \le 2 \le U$. Let $\pi$ be the distribution of $M$ under $H_1'$ and write $\nu_0$ and $\nu_1$ for the distribution of $Y$ under $H_0'$ and $H_1'$, respectively. \\
	Since the prior $M$ in \eqref{prior} consists of $k$ i.i.d.~scaled Rademacher vectors that each form $m/k$ columns of $M$ we have $\rank(M) \leq k$ and $\| M \|_\infty = u = 2\rho \leq a$ for $v$ small enough and since $k^{1/4}\sqrt{m/n}\leq a/2$. Thus $M \in \mathcal{A}(a,k)$.
	Then, reordering the columns of $M$ we have
	$$ \| M - \mathcal{A}(a,k_0) \|_F^2 = \| M_{ord} - \mathcal{A}(a,k_0)\|_F^2$$
	where $M_{ord}$ is a $m \times m$ matrix with the $(((i-1) m/k) +1) $-th to the  $(im/k)$-th columns each given by $u r_i$ where $r_i$ are $i.i.d$ Rademacher vectors of length $m$, $i=1,...,k$. Then (as in the proof of Theorem 1 in \cite{CarpentierNickl}) we transform $M_{ord}$ into the $m \times k$ matrix $M_{ord}P=u\sqrt{m/k}R$ consisting of $k$ column vectors $u \sqrt{m/k}r_i,~i=1,...,k$. The $m \times k$ projection matrix $P$ consists of $k$ column vectors, the $i$-th having zero entries except for the indices $s \in [((i-1) m/k) +1, \dots, im/k]$ where it equals $\sqrt{k/m}$. Hence $P$ is an orthonormal projection matrix and we obtain 
	$$\| M - \mathcal{A}(a,k_0) \|_F^2 \geq \| (M_{ord} - \mathcal{A}(a,k_0))P \|_F^2 = \| u\sqrt{m/k}R-\mathcal{A}(a\sqrt{m/k},k,k_0) \|_F^2$$
	where we define $$\mathcal A(a,k,k_0):= \{ A \in \mathbb{R}^{m \times k}: ~\| A \|_\infty \leq a ~~\text{and}~
	~\text{rank}(A)\leq k_0 \}.$$
	Therefore, if $\sigma_{min}(A)$ denotes the minimal singular value of a matrix $A$, we have that
	\begin{align}
	\| M - \mathcal A(a,k_0) \|_F^2
	& \geq \frac{m^2}{k} \| uR/\sqrt m - \mathcal A(a/ \sqrt{m}, k,k_0) \|_F^2 \notag \notag  \\
	& \geq   \frac{m^2u^2}{k} (k-k_0) (\sigma_{min}(R/\sqrt{m}))^2  \notag\\
	& \geq \frac{m^2u^2}{2} (\sigma_{min}(R/\sqrt{m}))^2  \geq \frac{m^2u^2}{4} = m^2 \rho^2
	\end{align}
	with probability going to 1, where we have used that $k-k_0 \geq k/2$ for $m$ large enough (recall $k_0=o(k)$) as well as the variational characterisation of minimal eigenvalues combined with Corollary 1 in \cite{NicklVandegeer} (with choices $n=m,~ p=k_1=k, ~\theta=0$ and $\Lambda_{min}=1$ there) to lower bound $\sigma^2_{min}(R/\sqrt m)$ by $1/2$. \\
	To conclude, $\pi$ is concentrated on $H_1$ and
	the primed testing problem above 
	is, asymptotically, strictly easier than the testing problem~\eqref{test1} since $H_0'$ is contained in $H_0$ and $H_1'$ is asymptotically contained in $H_1$. Thus, we have  for any test $\Psi$ by a standard lower bound (as, e.g., in (6.23) in \cite{Ginenickl}) that for all $\eta>0$
	$$\mathbb E_{H_0} \Psi + \sup_{H_1}\mathbb E_{H_1} (1-\Psi) \geq \mathbb E_{H_0'} \Psi + \mathbb E_{H_1'} (1-\Psi) - o(1) \geq (1-\eta)\left( 1- \frac{d_{\chi^2}(\nu_0, \nu_1)}{\eta}\right)-o(1),$$
	where $d_{\chi^2}(\nu_0, \nu_1)$ denotes the $\chi^2$-distance between $\nu_0$ and $\nu_1$, which remains to be bounded. \\ \\
	\textbf{Step II : Expectation over censored data}\\
	We define $I = [m]\times [m]$ and observe that the likelihood of the data under $\nu_0$ is
	$$L(Y_1,...Y_{m,m}) = \prod_{(i,j) \in I} \Big( (1-p)\mathbf 1_{ \{ Y_{ij} = 0\}} + \frac{p}{2} \mathbf 1_{ \{Y_{ij} = 1\}}+\frac{p}{2} \mathbf 1_{ \{Y_{ij} = -1 \}}\Big)$$
	and that the likelihood of the data under $\nu_1$ is
	$$L(Y_1,...Y_{m,m}) = \mathbb E_{M \sim \pi}\prod_{(i,j) \in I} \Big( (1-p)\mathbf 1_{ \{ Y_{ij} = 0 \}} + p(1/2 + M_{ij}/2)\mathbf 1_{ \{Y_{ij} = 1 \}} + p(1/2 - M_{ij}/2)\mathbf 1_{ \{Y_{ij} = -1 \}} \Big) .$$
	Thus, the likelihood ratio $\mathcal{L}$ between these two distributions is given by
	$$\mathcal{L} =  \mathbb E_{M\sim \pi} \prod_{(i,j) \in I} \Big( \mathbf 1_{ \{Y_{ij} = 0\}} + (1 + M_{ij})\mathbf 1_{ \{Y_{ij} = 1 \}} + (1 - M_{ij})\mathbf 1_{ \{Y_{ij} = -1 \}} \Big).$$
	So we have that
	\begin{align}\label{2}
	d_{\chi^2}(\nu_0, \nu_1)^2+1&  = \mathbb E_{Y \sim \nu_0} \mathcal{L}^2\nonumber\\
	& = \mathbb E_{Y\sim \nu_0}\Big[\mathbb E_{M \sim \pi} \prod_{(i,j) \in I} \Big( \mathbf 1_{ \{Y_{ij} = 0 \}} + (1 + M_{ij})\mathbf 1_{ \{Y_{ij} = 1 \}} + (1 - M_{ij})\mathbf 1_{ \{Y_{ij} = -1 \}} \Big)\Big]^2\nonumber\\
	&= \mathbb E_{M,M' \sim \pi} \prod_{i,j} \Big[\Big(1-p+ \frac{p}{2} (1 + M_{ij}) (1 + M_{ij}') + \frac{p}{2}(1 - M_{ij})  (1 - M_{ij}') \Big)\Big]\nonumber\\
	&= \mathbb E_{M,M'\sim \pi} \prod_{i,j} \Big[1 +p M_{ij}M_{ij}'\Big].
	\end{align}
	where $M'$ is an independent copy of $M$. 
	
	\smallskip
	
	\textbf{Step III : Conditioning over the cross information} \\
	Let $N_{r,r'}$ be the number of times where the couple $K_j=r,K'_j=r'$ occurs. That is,
	\begin{equation*}
	N_{r,r'}:=  \sum_{j=1}^m \mathbf 1_{ \{ K_j=r, K_{j}'=r' \}}.
	\end{equation*}
	We enumerate the elements inside these groups from $1$ to $N_{r,r'}$. We write $\tilde V_j^{r,r'}$ for the corresponding enumeration of the $V_j$. Setting $\mathbf N = (N_{r,r'})_{r,r'}$ and using the definition of the prior, we compute
	\begin{align}\label{1}
	\mathbb E_{M,M' \sim \pi} \prod_{i,j} \Big[1 +p M_{ij}M_{ij}'\Big] &= \mathbb E_{\mathbf N,U,\tilde V,U',\tilde V'} \prod_{i=1}^m \prod_{r,r' \in \{1,\ldots, k\}^2} \prod_{j=1}^{N_{r,r'}} \Big[1 +pu^2 U_i^r \tilde V_j^{r,r'}(U_i^{r'})' (\tilde V_j^{r,r'})'\Big]\nonumber\\
	&=: \mathbb E_{\mathbf N}  \prod_{r,r' \in \{1,\ldots, k\}^2}  \mathcal I(N_{r,r'})
	\end{align}
	where we define for any $N = N_{r,r'} > 0$ 
	\begin{align*}
	\mathcal I(N) = \mathbb E_{X,W,X',W'} \prod_{i=1}^m \prod_{j=1}^{N} \Big[1 +pu^2 X_i W_j X_i' W_j'\Big]
	\end{align*}
	and where $(X_i)_{i \leq m}, (X'_i)_{i \leq m}, (W_i)_{j \leq N}, (W'_i)_{j \leq N}$ are $i.i.d$.~Rademacher random variables. Moreover, we set $\mathcal I_{r,r'}(0) = 0$. 
	
	\smallskip	
	
	\textbf{Step IV : Bound on $\mathbb E_{\mathbf N}  \prod_{r,r' \in \{1,\ldots, k\}^2}  \mathcal I(N_{r,r'})$.}\\
	In order to  bound $\mathcal I(N)$ we use the following lemma proved below
	\begin{Lemma}\label{lem:boundIn}
		Let $N=N_{r,r'}$. There exist constants $C_1, C_2, C_3 >0$ such that for $v$ small enough 
		\begin{align}\label{3}
		\mathcal I(N) 
		&\leq    \exp\Big (C_1 v^4N/m\Big)  \exp\Big(\frac{C_2 v^4k^2N}{m^2}\Big)\exp\Big(C_3v^4N^2k^2/m^2\Big).
		\end{align}
	\end{Lemma}
	Using \eqref{2}, \eqref{1} and \eqref{3} we have that 
	\begin{align}
	& d_{\chi^2}(\nu_0, \nu_1)^2+1\notag \\
	& = \mathbb E_{\mathbf N}  \prod_{r,r' \in \{1,\ldots, k\}^2}  \mathcal I(N_{r,r'})\label{4}\\
	&\leq  \mathbb E_{\mathbf N} \left [ \left (  \exp\left(\frac{C_2 v^4 k^2}{m^2}\sum_{r,r'} N_{r,r'}\right)\right)\left(\exp \left ( \dfrac{C_1v^4}{m}\sum_{r,r'} N_{r,r'}  \right ) \right )  \left (\prod_{r,r' \in \{1,\ldots, k\}^2}  \exp\Big(C_3v^4N_{r,r'}^2k^2/m^2\Big) \right ) \right ] \notag\\
	&=  \exp\left (C_2v^4\frac{k^2}{m}+ C_1 v^4 \right )\mathbb E_{\mathbf N}  \left [\prod_{r,r' \in \{1,\ldots, k\}^2}  \exp\Big(C_3v^4N_{r,r'}^2k^2/m^2\Big)\right ], \label{Stochterm1}
	\end{align}
	since $\sum_{r,r'} N_{r,r'} = m$. We bound the expectation of the stochastic term in \eqref{Stochterm1} using the following lemma proved below:
	\begin{Lemma}\label{lem:expmN}
		There exists a constant $C'>0$ such that for $v$ small enough we have 
		\begin{align}\label{5}
		&\mathbb E_N \Big[\prod_{r,r'} \exp\Big(C_3v^4N_{r,r'}^2k^2/m^2\Big)\Big] \leq 1 + 2C'v^4 + \exp\Big(-m/k^2\Big).
		\end{align}
	\end{Lemma}
	\noindent Inserting \eqref{5} into \eqref{Stochterm1} and summarizing all the steps we obtain 
	\begin{align*}
	0 \leq  d_{\chi^2}(\nu_0, \nu_1)^2 \leq C \left (v^2 + \exp\big(-m/k^2\big) \right ) =o(1)
	\end{align*}
	for a constant $C > 0$ and therefore, letting $\eta \to 0$,
	$$ \mathbb{E}_0 [ \Psi] + \sup_{H_1} \mathbb{E}_{H_1} [1- \Psi] \geq (1-\eta)\left( 1- \frac{d_{\chi^2}(\nu_0, \nu_1)}{\eta}\right)-o(1) = 1 - o(1).$$  
	
\end{proof}

\begin{proof}[Proof of Lemma~\ref{lem:boundIn}]
	Note that, by construction of $\mathcal P$, we have that $$N=N_{r,r'} \leq m/k$$ since the number of $j$ where $M_{.,j}$ corresponds to $K_j=r$ is bounded by $m/k$. As the product of two independent Rademacher random variables is again a Rademacher random variable, we have
	\begin{align*}
	\mathcal I(N)=  \mathbb E_{R,R'} \prod_{i=1}^m  \prod_{j=1}^{N} \Big[1 +pu^2 R_i R_j'\Big],
	\end{align*}
	where $R=(R_{i})_{i=1}^{m},R'=(R'_{i})_{i=1}^{N}$ are independent Rademacher vectors of length  $m$ and $N$, respectively. The usual strategy to use $1+x \le e^x$ and then to bound iterated exponential moments of Rademacher variables (as in the proof of Theorem 1 of \cite{CarpentierNickl}) only works when $k=const$, and a more refined estimate is required for growing $k$, as relevant here.  \\ 
	We now bound $\mathcal I(N)$ for a fixed $N, m/k \geq N > 0$. Using the binomial theorem twice we have
	\begin{align*}
	\mathcal I(N) 
	&= \mathbb E_{R'}\Bigg[ \Big[\frac{1}{2}\prod_{j=1}^{N} \big[1 +pu^2  R_j'\big] + \frac{1}{2}\prod_{j=1}^{N} \big[1 -pu^2  R_j'\big] \Big]^m\Bigg]\\
	&=  \frac{1}{2^m}\sum_{s=1}^m {m \choose s} \Big[\frac{1}{2}\big[1 +pu^2\big]^s \big[1 -pu^2\big]^{m-s} + \frac{1}{2}\big[1 -pu^2  \big]^s \big[1 +pu^2 \big]^{m-s} \Big]^N\\
	&=  \frac{1}{2^m2^N}\sum_{s=1}^m {m \choose s}\sum_{q=1}^N {N \choose q} \big[1 +pu^2\big]^{sq +(m-s)(N-q)} \big[1 -pu^2\big]^{(m-s)q + s(N-q)}\\
	&=  \mathbb E_{Q,S} \Big[\big[1 +pu^2\big]^{SQ +(m-S)(N-Q)} \big[1 -pu^2\big]^{(m-S)Q + S(N-Q)}\Big]
	\end{align*}
	with independent Binomial random variables $S \sim \mathcal B(1/2, m), Q \sim \mathcal B(1/2, N)$. 
	If $A := \frac{1 -pu^2}{1+pu^2}$, we obtain
	\begin{align*}
	\mathcal I(N) &=  \mathbb E_{Q,S} \Big[\big[1 +pu^2\big]^{mN} \left [\frac{1 -pu^2}{1+pu^2}\right ]^{SN+mQ-2SQ}\Big]\\
	&=  \big[1 +pu^2\big]^{mN}\mathbb E_{Q} \Big[A^{mQ} \mathbb E_{S}A^{S(N-2Q)}\Big]\\
	&=  \big[1 +pu^2\big]^{mN}\mathbb E_{Q} \left [A^{mQ} 2^{-m}\left (A^{(N-2Q)}+1\right )^{m}\right ]\\
	&=  \big[1 +pu^2\big]^{mN}\mathbb E_{Q} \left [A^{Nm/2}\left (\frac{1}{2}A^{(N/2-Q)}+\frac{1}{2}A^{(-N/2+Q)}\right )^{m}\right ]\\
	&=  \big[1 - p^2u^4\big]^{mN/2}\mathbb E_{Q}  \left (\frac{1}{2} A^{Q-N/2}+ \frac{1}{2}A^{N/2-Q}\right )^m.
	\end{align*}
	Now, we denote $x := pu^2 = 4vk^{1/2}/m \leq 1/2$ for $v$ small enough. Furthermore, we Taylor expand $\log (A)$ about $1$ up to second order, i.e.
	$$ \log(A)= \log(1-x)-\log(1+x) = - 2x - \frac{1}{2} \left ( \frac{1}{\xi_1^2}-\frac{1}{\xi_2^2}\right )x^2=:-2x-c(x)x^2$$
	for $\xi_1 \in [1/2,1]$, $\xi_2 \in [1,3/2]$ and where $c(x) \in [0,16/9 ]$	
	since $x \leq 1/2$. Hence, using also the inequality $e^{x}\leq 1+x+x^{2}/2+x^{3}/6+2x^{4}$ we deduce
	\begin{align*}
	\mathcal I(N) 
	&\leq  \exp\big[-mNx^2/2\big] \mathbb E_Q \Big[ \frac{1}{2} \exp\big(-2x(Q-N/2) -c(x)(Q-N/2)x^2)\big)\\
	&~~~+  \frac{1}{2} \exp\big(-2x(N/2-Q) -c(x)(N/2-Q)x^2)\big) \Big]^m\\
	&\leq  \exp\big[-mNx^2/2\big] \\ & ~~~\cdot\mathbb E_Q \Bigg[ \frac{1}{2} \Big(1-2x(Q-N/2) -c(x)(Q-N/2)x^2 + (-2x(Q-N/2) -c(x)(Q-N/2)x^2)^2/2\\ 
	&+ (-2x(Q-N/2) -c(x)(Q-N/2)x^2)^3/6 + 2(-2x(Q-N/2) -c(x)(Q-N/2)x^2)^4\Big)\\
	&+  \frac{1}{2}\Big(1-2x(N/2-Q) -c(x)(N/2-Q)x^2 + (-2x(N/2-Q) -c(x)(N/2-Q)x^2)^2/2\\ 
	&+ (-2x(N/2-Q) -c(x)(N/2-Q)x^2)^3/6 + 2(-2x(N/2-Q) -c(x)(N/2-Q)x^2)^4\Big)\Bigg]^m.
	\end{align*}
	Since $x \leq 1/2$ and $|N/2 - Q|x \leq 1/4$ there exist two constants $c_2=c_2(x)=c(x)/2+c(x)^2/32 \leq 1$ and $ c_1 = c_1(x) = 32+32c(x)+12c(x)^2+2c(x)^3+c(x)^4/8 \leq 140 $ such that the last equation above can be bounded by
	\begin{align*}
	&\leq  \exp\big[-mNx^2/2\big] \mathbb E_Q \Big[ 1 +  2 x^2 (Q-N/2)^2 + c_1|Q-N/2|^4x^4 + c_2|Q-N/2|x^2  \Big]^m\\
	&\leq  \exp\big[-mNx^2/2\big] \mathbb E_Q \exp\Big[   mx^2 (N-2Q)^2/2 + c_1m(Q-N/2)^4x^4 +c_2m|Q-N/2|x^2 \Big]\\
	&=  \mathbb E_Q \Big[\exp\Big(\frac{m}{2}\big( x^2 (2Q-N)^2 - Nx^2\big)\Big) \exp\Big(c_1m(Q-N/2)^4x^4 + c_2m|Q-N/2|x^2\Big)\Big].
	\end{align*}
	Using the Cauchy-Schwarz inequality twice, this implies that
	\begin{align*}
	\mathcal I(N) 
	&\leq   \sqrt{\mathbb E_Q \Big[\exp\Big(mx^2N\big( (2Q-N)^2/N - 1\big)\Big)\Big]} \Bigg[\mathbb E_Q \Big[\exp\Big(c_1 m x^4(N-2Q)^4/4\Big)\Big] \\ &~~ \cdot \mathbb E_Q \Big[\exp\Big(2c_2m|2Q-N|x^2\Big)\Big] \Bigg]^{1/4} =: \sqrt{(I)} (II)^{1/4} (III)^{1/4}.
	\end{align*}
	\noindent
	{\bf Step 1 : Bound on term $(III)$} \\
	Since $Q \sim \mathcal{B}(1/2,N)$, since $(2Q-N)$ is symmetric  and since $2c_2mx^2\leq1/2$ we have that 
	\begin{align} (III) & =\mathbb E_Q \Big[\exp\Big(2c_2 m|2Q-N|x^2\Big)\Big]  \leq 2\mathbb E_Q \Big[\exp\Big(2c_2 m(2Q-N)x^2\Big)\Big]\nonumber\\
	&=2 \Big[\exp\Big(2c_2 m x^2\Big)+ \exp\Big(-2c_2 m x^2\Big)\Big]^N \leq 2 \Big[1+ 8c_2^2 m^2 x^4\Big]^N \nonumber\\
	& \leq \exp \big (8c_2^2 m^2 x^4 N\big ) \leq \exp \big ( \frac{C_2 v^4 k^2N}{m^2}\big ).\label{eq:iii}
	\end{align}

	\noindent
	{\bf Step 2 : Term $(II)$} \\ 
	We use $mN^2x^4 \leq 64v^4/m$,  $(N-2Q)^2 \leq N^2$ and $N\leq m/k$ to obtain
	$$(II) \leq \mathbb{E}_Q \Big[\exp\Big(64c_1 v^4N/m\cdot (N-2Q)^2/N\Big)\Big].$$ 
	Since $Q\sim \mathcal B(1/2,N)$ the Rademacher average  $Z=(N-2Q)/\sqrt{N}$ is sub-Gaussian with sub-Gaussian constant at most $1$. It  hence satisfies (e.g., equation (2.24) in \cite{Ginenickl}) for $c > 2 $ $$\mathbb{E} \exp\{Z^2/c^2\} \le 1 + \frac{2}{c^2/4-1} \leq e^{c_3 c^{-2}},$$ which for $v$ small  enough and the choice $c^{-2}=64c_1 v^4N/m$ implies for some constant $C_1$ that $$(II) \leq \exp \left (\frac{4C_1 v^4N}{m}\right ).$$
	
	\noindent 
	{\bf Step 3 : Term $(I)$} \\
	We have that
	\begin{align*}
	(I)&= \mathbb E_Q \Big[\exp\Big(mNx^2\left[ \frac{(2Q-N)^2}{N} - 1\right]\Big)\Big]\\
	&= \mathbb E \left [\exp\left ( \frac{16v^2Nk}{m}\left [ \frac{1}{N}\left(\sum_{i=1}^N \eps_i\right )^2 - 1\right ]\right )\right ] = \mathbb E \left [\exp\left (\frac{16v^2k}{m} \sum_{i\neq j, i,j \leq N} \eps_i\eps_j \right)\right ],
	\end{align*}
	where  $\eps_i$ are $i.i.d.$~Rademacher random variables. If $A=(a_{ij})$ is a symmetric matrix with all elements on the diagonal equal to zero, then for the Laplace transform of an order-two Rademacher chaos $Z=\sum_{i, j} a_{ij} \eps_i \eps_j$ we have the inequality $$\mathbb{E}e^{\lambda Z} \le \exp \left\{\frac{16 \lambda^2 \|A\|_F^2}{2 \left(1-64\|A\|\lambda \right)} \right\}, ~~\lambda>0,$$ see, e.g., Exercise 6.9 on p.212 in \cite{Boucheronetal} with $\mathcal T=\{A\}$. Now take  $A=(\delta_{i \neq j})_{i,j \leq N}$ so that we have $\| A \| \leq N$ and for $v$ small enough $16v^2kN/m \leq 16v^2 \leq 1/128$.
	\begin{align*}
	& \mathbb E \Big[\exp\Big(\frac{16v^2k}{m}\sum_{i\neq j, i,j \leq N} \eps_i\eps_j \Big)\Big] \leq  \exp \left ( \frac{16^3 v^4k^2 \|A \|_F^2 }{2m^2(1-1024v^2k\|A \|/m)}\right ) \leq  \exp \left ( \frac{16^3v^4k^2N^2}{m^2}\right ) 
	\end{align*}
	and therefore we conclude for a constant $C_3 > 0$ that 
	\begin{align}\label{eq:i}
	(I)
	&\leq \exp\Big(2C_3v^4k^2N^2/m^2 \Big).
	\end{align}
	\noindent
	{\bf Step 4 : Conclusion on $\mathcal I(N)$} \\
	Combining the bounds for $(I)$, $(II)$ and $(III)$ with the bound on $\mathcal{I}(N)$ we have that 
	\begin{align*}
	\mathcal I(N) 
	&\leq\exp\Big(C_2 v^4 k^2N/m^2\Big) \exp\Big(C_1 v^4N/m\Big)\exp\Big(C_3v^4k^2N^2/m^2\Big).
	\end{align*}
\end{proof}

\begin{proof}[Proof of Lemma~\ref{lem:expmN}] We bound the expectation by bounding it separately on two complementary events.
	For this we consider the event $\xi$ where all $N_{r,r'}$ are upper bounded by $\tau:= 15 m/k^2$, assumed to be an integer (if not replace it by its integer part plus one in the argument below). More precisely we define 
	$$\xi = \Big\{\forall r \leq k, \forall r' \leq k\;:\; N_{r,r'} \leq \tau \Big\}.$$
	Note that $\{N_{r,r'} >\tau\}$ occurs only if the size of the intersection of the class $r$ of partition $\mathcal P$ with the class $r'$ of partition $\mathcal P'$ is larger than $\tau$. This means that at least $\tau$ elements among $m/k$ elements of the class $r'$, must belong to the class $r$. The positions of these $\tau$ elements can be taken arbitrarily within the $m/k$ elements. For the first element,   among those $\tau$, the probability to belong to the class $r$ is $\frac{m/k}{m}$. For the second element this probability is $\frac{m/k}{m-1}$ or $\frac{(m/k)-1}{m-1}$ and so on.  All these probabilities are smaller than $(m/k)/ (m-m/k+1)$.  Therefore we have
	$$\mathbb P_{\mathbf N}(N_{r,r'} > \tau) \leq {m/k \choose \tau} \left (\frac{m/k}{m - m/k+1}\right )^\tau \leq \frac{(m/k)^\tau}{\tau!} (2/k)^{\tau} \leq 2^\tau(m/k^2)^\tau \tau^{-\tau} e^\tau \le e^{-\tau},$$
	where we use  ${m/k \choose \tau}\leq \frac{(m/k)^\tau}{\tau!}$ and Stirling's formula.
	Using a union bound this implies that the probability of $\xi$ is lower bounded by $1-k^2\exp(-15m/k^2)$. \\
	We have on the event $\xi$
	\begin{align*}
	&\mathbb E_{\mathbf N} \Big[\mathbf 1\{\xi\} \prod_{r,r' \in \{1,\ldots, k\}^2}  \exp\Big(C_3v^4N_{r,r'}^2k^2/m^2\Big)\Big]\\
	&\leq \exp\Big(C_3v^4 k^2 \cdot 15^2  (m/k^2)^2 k^2/m^2\Big)\Big]\\
	&\leq\exp\Big(C'v^4\Big) \leq 1 + 2C'v^4.
	\end{align*}
	for $C'=225C_3$ and for $v$ small enough. 
	Moreover,  by definition of $N_{r,r'}$, we have that $N_{r,r'}\leq m/k$ and $\sum_{r,r'} N_{r,r'}=m$. Hence
	$$\sum_{r,r'}N_{r,r'}^2\leq km^2/k^2 = m^2/k$$ 
	which implies that on $\xi^C$
	\begin{align*}
	&\mathbb E_{\mathbf N} \Big[\mathbf 1\{\xi^C\} \prod_{r,r' \in \{1,\ldots, k\}^2}  \exp\Big(C_3v^4N_{r,r'}^2k^2/m^2\Big)\Big]\\
	&\leq \mathbb P_{\mathbf N}(\xi^C)  \exp\Big(C_3v^4k\Big)\\
	&\leq k^2  \exp\Big(-15m/k^2 + C_3v^4k\Big)\\
	&\leq k^2  \exp\Big(-3m/k^2\Big) \leq \exp\Big(-m/k^2\Big),
	\end{align*}
	for $v$ small enough and since $k^3 \leq m$. 
	Thus, combining the bounds on $\xi$ and $\xi^C$, we have that
	\begin{align*}
	&\mathbb E_N \Big[\prod_{r,r'} \exp\Big(C_3v^4N_{r,r'}^2k^2/m^2\Big)\Big] \leq 1 + 2C'v^4 + \exp\Big(-m/k^2\Big).
	\end{align*}
\end{proof}


\section{Auxiliary results \label{sec:Proofs2}}
\subsection{Proof of Lemma \ref{Kombinatorikinequality}}

\begin{proof}
	Assume that among the first $n/4$ samples we have less than $n/8$ entries that are sampled twice - otherwise the result holds since $n/8 \geq n^2/64m_1m_2$ for $n \leq m_1 m_2$. Then, among the first $n/4$ samples, there are at least $n/8$ distinct elements of $\mathcal{B}$, the set of all standard basis matrices in $\mathbb{R}^{m_1 \times m_2}$, that have been sampled at least once. We write $\mathcal S$ for the set of \textit{distinct} elements of $\{X_i\}_{i \leq n/4}$ and obviously have  $|\mathcal S| \geq n/8$. Hence, by definition of the sampling scheme, we have that 
	$$\mathbb P(X_i \in \mathcal S) \geq \frac{n}{8 m_1 m_2}, ~~~~n/4<i \leq n/2.$$
	{Furthermore, when sampling an element from $\mathcal{S}$ we have to remove this element from $\mathcal{S}$ as we have to use the entry that is stored in $\mathcal{S}$ to form a pair of entries. Hence the probability to sample another element from $\mathcal{S}$ decreases and is bounded by }
	$$ \mathbb{P} ( X_{j}  \in \mathcal{S} \backslash \{X_i \} \big | X_i \in \mathcal{S} ) \geq  \frac{n-1}{8m_1 m_2}$$
	for $n/4 < i < j < n/2$. We deduce by induction for $j > i+k$ and $k \leq n/2-i-1$ that 
	$$ {\mathbb{P} (X_j \in \mathcal{S} \backslash \{ X_i,...,X_{i+k} \} \big | X_i,...,X_{i+k} \in \mathcal{S} )\geq \frac{n-k}{8m_1 m_2}}$$
	which yields
	
	\begin{align}\label{lemma1_1}
	\mathbb{P} \left ( N \geq \frac{n^2}{64 m_1 m_2}\right ) & \geq
	\mathbb{P}\left (\sum_{n/4<i \leq n/2} \mathbf 1_{\{X_i \in \mathcal S\} } \geq \frac{n^2}{64 m_1 m_2} \right ) \notag \\& \geq \mathbb{P}\left (\sum_{n/4<i \leq n/2} \mathbf Z_{i}\geq \frac{n^2}{64 m_1 m_2}\right )
	\end{align}
	where $\mathbf Z_{i}$ can be taken to be Bernoulli random variables with success probability
	$$p'=\frac{n-\frac{n^2}{64 m_1 m_2}}{8 m_1 m_2}.$$
	Then, Bernstein's inequality for bounded random variables (see e.g. Theorem 3.1.7 in \cite{Ginenickl}), \eqref{lemma1_1} and the estimates 
	$${\mathbb{E} \left [\sum_{n/4<i \leq n/2} \mathbf Z_{i} \right ]\geq \frac{n^2}{33 m_1 m_2} }$$ 
	{which holds for $n \leq m_1 m_2$ and }$${ \Var  \left (\sum_{n/4<i \leq n/2} \mathbf Z_{i}\right ) \leq \frac{n^2}{32 m_1 m_2}} $$
	imply that
	\begin{align*} \mathbb{P} \left ( N \geq \frac{n^2}{64m_1m_2} \right ) & \geq  1 - \mathbb{P} \left ( \sum_{n/4<i \leq n/2} \mathbf Z_{i} - \mathbb{E} \left [ \sum_{n/4<i \leq n/2} \mathbf Z_{i}\right ]\leq \frac{-n^2}{72 m_1 m_2} \right )  \geq 1 - \exp \left ( \frac{n^2}{372m_1 m_2} \right ).\end{align*}
\end{proof}

\subsection{Lemma \ref{lem_bernoulli_rip} 
}

%
\begin{Lemma} \label{lem_bernoulli_rip}
	Consider the Bernoulli model \eqref{Bernoulli_model} and
	assume $n \geq m \log(d)$. Then, 
	with probability at least $1-2\exp(-100d)$ we have for any given $M \in \mathcal{A}(a,m)$ that
	$$\underset{A\in\mathcal{A}(a,m), ~\|M-A\|_F \geq Ca\sqrt{(\rank(A) \vee 1)d/p}}{\sup} \left[ \left \vert 
	\sum_{i,j} (B_{ij}-p)(A_{ij}-M_{ij})^2 \right |	- \frac{p}{2} \|M_0 - A \|_F^2 
	\right]\leq 0$$
	where $C=8072$.
\end{Lemma}
\begin{proof}
	We have, using a union bound, that
	\begin{align}\label{lemma_bernoulli_rip_1}
	& \mathbb{P} \left( \underset{A\in\mathcal{A}(a,m), ~\|M-A\|_F \geq Ca\sqrt{(\rank(A) \vee 1)d/p}}{\sup}
	\left[\left \vert 
	\sum_{i,j} (B_{ij}-p)(A_{ij}-M_{ij})^2 \right |-\frac{p}{2}\Vert M_0-A\Vert _F^2 \right]> 0\right )\nonumber\\
	\leq  \sum_{k=1}^m ~& \mathbb{P} \left(
	\underset{A\in\mathcal{A}(a,k),~p\|M-A\|_F^2 \geq C^2a^2kd}{\sup}
	\left[\left \vert 
	\sum_{i,j} (B_{ij}-p)(A_{ij}-M_{ij})^2 \right |-\frac{p}{2}\Vert A-M\Vert_F^2\right]> 0\right ).
	\end{align}
	Then, using a peeling argument each of the terms in \eqref{lemma_bernoulli_rip_1} can be bounded by
	\begin{align}
	&\sum _{s \in \mathbb{N}:~C^2a^2kd/2 \leq 2^s < \infty } \mathbb{P} \left(\underset{A\in \mathcal{A}(a,k), ~2^s  \leq p\| A - M \|_F^2 \leq 2^{s+1} }{\sup}\left \vert 
	\sum_{i,j} (B_{ij}-p)(A_{ij}-M_{ij})^2 \right | > 2^{s}/2\right ) \notag \\\leq & \sum _{s \in \mathbb{N}:~C^2a^2kd/2 \leq 2^s < \infty }  \mathbb{P} \left(\underset{A \in \mathcal{A}(a,k), ~p\|A-M \|_F^2 \leq 2^{s+1}}{\sup}\left \vert 
	\sum_{i,j} (B_{ij}-p)(A_{ij}-M_{ij})^2 \right | > 2^{s}/2 \right )  \label{II}
	\end{align}
	with the convention that if the supremum is taken over an empty set the corresponding probability is set equal to $0$. For the cases where the supremum is not taken over an empty set, we apply
	Lemma \ref{Lemma:Talagrand} (with choices $\xi_{ij}=1,$ $q=2$, $z=4$, $U=1$ and $t=2^s$ there ) and obtain for
	\begin{align*} Z(s): 
	& = \sup_{A \in \mathcal{A}(a,k),~p\| A-M \|_F^2  \leq 2^{s+1} } \left | \sum_{i,j} (B_{ij}-p)(A_{ij}-M_{ij})^2 \right | \end{align*}
	that we can bound
	$$\mathbb{P} \left ( Z(s)  > 2^s/2 \right ) \leq \exp \left ( \frac{-2^s}{260352a^2
	}\right )$$
	Hence,  \eqref{II} can be upper bounded by
	\begin{align}
	\sum_{s \in \mathbb{N}:~Ca^2kd/2 \leq 2^s < \infty}  \exp \left ( \frac{-2^{s}}{260352a^2} \right ) \notag  \leq  & 2 \exp \left ( -\frac{C^2kd}{260352} \right )  \leq  2 \exp (-101 d).  \notag
	\end{align}
	The result then follows by noting that $\log(m) \leq d$.
\end{proof}

\subsection{Lemma \ref{Lemma:Talagrand}}
\begin{Lemma} \label{Lemma:Talagrand} Consider the Bernoulli model (\ref{Bernoulli_model}). 
	Suppose that $\xi_{ij}$ are independent random variables with  $\max_{ij} |\xi_{ij} | \leq U$ and that $m \log(d) \leq n$. Let $z > 0$, $q\in \{1,2\}$, $M \in \mathcal{A}(a,m)$ and $1 \leq k_0 < m$ be given. Finally, for $C=1009$ suppose that $t \in \mathbb{R}_+$ is such that $t \geq C^2z(4a)^{2q-2}U^2k_0d/2$ and that the supremum in
	$$Z(t):= \sup_{A \in \mathcal{A}(a,k_0),~ p\|A-M \|_F^2  \leq 2t} \left | \sum_{i,j} \left [(B_{ij} \xi_{ij}-\mathbb{E} B_{ij} \xi_{ij})(A_{ij}-M_{ij})^q  \right ] \right |$$ 
	is not empty. Then,
	\begin{equation} \label{Bound:Z(s)}
	\mathbb{P} \left ( Z(t) > \frac{t}{\sqrt{z}}\right ) \leq \exp \left ( \frac{-t}{
		32^2 ( 8(2a)^{2q-2}U^2z + 505(2a)^{q}U \sqrt{z}/32 )
	} \right )
	\end{equation} 
\end{Lemma}
\begin{proof}
	We first bound $\mathbb{E}Z(t)$ and then apply Talagrand's \cite{TalagrandInventiones} inequality. 
	Using symmetrization (e.g. Theorem 3.1.21 in \cite{Ginenickl}) and two contraction inequalities (e.g. Theorems 3.1.17 and 3.2.1 in \cite{Ginenickl}), we obtain that
	\begin{align}
	\mathbb{E} Z(t)  & \leq 2U \mathbb{E} \left ( \sup_{A \in \mathcal{A}(a,k_0),~p\| A-M\|_F^2 \leq 2t } \left | \sum_{i,j} B_{ij} \eps_{ij}(A_{ij}-M_{ij})^q \right | \right ) \notag \\
	& \leq  2(4a)^{q-1} U\mathbb{E} \left ( \sup_{A \in \mathcal{A}(a,k_0),~p\| A-M\|_F^2 \leq 2t} \left | \sum_{i,j} B_{ij} \eps_{ij}(A_{ij}-M_{ij}) \right | \right ) \notag \\
	&	\leq 2(4a)^{q-1}U\bE\left (\underset{ A \in \mathcal{A}(a,k_0), ~p \|A-M\|_F^2 \leq 2t}{\sup}\left \vert \left\langle \Sigma_R,A-A_0\right\rangle\right \vert\right )\notag + 2(4a)^{q-1}U \mathbb{E} \left | \langle \Sigma_R, A_0-M \rangle \right | \\&
	\leq 8(4a)^{q-1}U\sqrt{k_0t/p}\bE \left\Vert \Sigma_R\right\Vert + 2(4a)^{q-1}U \mathbb{E} \left | \langle \Sigma_R, A_0-M \rangle \right |.  \label{Generalbound: expectationbound1}
	\end{align}
	where $\eps_{ij}$ are independent Rademacher random variables, 
	$\Sigma_R := \big (B_{ij}\eps_{ij}  \big )_{ij}$ and where $A_0$ is an arbitrary element in $\mathcal{A}(a,k_0)$ such that $p\| A_0-M \|_F^2 \leq 2t$. Such an $A_0$ exists as soon as the supremum is not taken over an empty set.
	An extension of Corollary 3.6 in \cite{Bandeira} to rectangular matrices by self-adjoint dilation (e.g. section 3.1. in \cite{Bandeira}) implies (with choices $\xi_{ij}=B_{ij} \eps_{ij} / \sqrt{p}$, $b_{ij}= \sqrt{p}$, $\alpha=3$ and $\sigma= \max \left ( \max_{j} \sqrt{\sum_{i=1}^{m_1} b_{ij}^2},  \max_i \sqrt{\sum_{j=1}^{m_2} b_{ij}^2 } \right ) \leq \sqrt{pd}$ there )  that 
	\begin{equation*}
	\bE\left\Vert \Sigma_R \right\Vert\leq e^{2/3} (2 \sqrt{pd} + 42  \sqrt{\log(d)}) \leq86 \sqrt{pd}
	\end{equation*}
	since $m \log(d) \leq n$. For the second term in \eqref{Generalbound: expectationbound1} we have
	\begin{align*} \mathbb{E} | \langle \Sigma_R, A_0-M \rangle |& \leq \left ( \Var ( \langle \Sigma_R, A_0 - M\rangle ) \right )^{1/2} \\
	& = \left  ( p \|A_0-M \|_F^2 \right )^{1/2} \leq \sqrt{2t} .
	\end{align*} 
	Hence, for $C^2z(4a)^{2q-2}U^2 k_0d/2 \leq t$ and since $C = 1009$ we have that 
	\begin{equation} \label{Generalbound: Expectationbound2}
	\mathbb{E}Z(t) \leq 688(4a)^{q-1}U\sqrt{k_0td } + 2(4a)^{q-1}U\sqrt{2t} \leq  31  t/(32 \sqrt{z}).
	\end{equation}
	We now make use of the following inequality due to Talagrand \cite{TalagrandInventiones}, which in the current form with explicit constants can be obtained by inverting the tail bound in Theorem 3.3.16 in \cite{Ginenickl}. 
	\begin{Theorem} \label{Talagrand} Let $(S, \mathcal{S})$ be a measurable space and let $n \in \mathbb{N}$. 
		Let $X_k,~k=1,\dots,n  $ be independent ${S}$-valued random variables and let $\mathcal{F}$ be a countable set of functions $f=(f_1,...,f_n):{S}^n \rightarrow [-K,K]^n$ such that $\mathbb{E}f_k(X_k)=0$ for all $f\in \mathcal{F}$ and $k=1,...,n$. Set
		$$Z:=\sup_{f \in \mathcal{F}} \sum_{k=1}^n f_k(X_k) .$$
		Define the variance proxy
		$$V_n:=2K\mathbb{E}Z + \sup_{f \in \mathcal{F} }\sum_{k=1}^n \mathbb{E} \left [(f_k(X_k))^2 \right ].$$
		Then, for all $t \geq 0$,
		\begin{equation*}
		\mathbb{P} \left ( Z - \mathbb{E} Z \geq t \right ) \leq \exp \left ( \frac{-t^2}{4V_n + (9/2)Kt} \right ).
		\end{equation*}
	\end{Theorem}
	The functional $A \to \|A-M\|_F^2$ is continuous on the compact set of matrices  $\{ A \in \mathcal{A}(a,k_0):$ $\|A-M\|_F^2 \leq 2t \} $, hence by continuity and compactness the supremum is attained over a countable subset. 
	Thus we may apply Talagrand's inequality to $Z(t)$. We have for our particular case,  since  $\sup_{f \in \mathcal{F}} |f(X)|=\sup_{f \in \{ \mathcal{F} \bigcup\{ -\mathcal{F}\}\}} f(x)$, that
	$$X_{ij}=B_{ij}\xi_{ij}-\mathbb{E} B_{ij}\xi_{ij}, ~~~S=[-2U,2U ]$$
	\begin{align*}
	\mathcal{F}= \bigg \{ &f: S^{m_1 \times m_2} \rightarrow [-2(2a)^qU,2(2a)^q U]^{m_1 \times m_2}, ~f_{ij}(X_{ij})=(-1)^lX_{ij}(A_{ij}-M_{ij})^q, \\
	&	A \in \mathcal{A}(a, k_0), ~p \|A - M \|_F^2 \leq 2t,~ l \in \{0,1 \} \bigg \}
	\end{align*}
	and moreover
	\begin{align*} & \sup_{(A, l), ~A \in \mathcal{A}(a,k_0), ~p\|A - M \|_F^2\leq 2t, ~l \in \{ 0,1\}}
	\sum_{i,j} \mathbb{E} \left [ \left ( (-1)^l(B_{ij}\xi_{ij}- \mathbb{E} B_{ij} \xi_{ij}) (A_{ij}-M_{ij})^q \right )^2  \right ]\\ 
	\leq & (2a)^{2q-2} \sup_{A \in \mathcal{A}(a,k_0),~ p\|A - M \|_F^2\leq 2t}
	\sum_{i,j} \Var (B_{ij}\xi_{ij}) (A_{ij}-M_{ij})^2 \\
	\leq &  (2a)^{2q-2}U^2 \sup_{A \in \mathcal{A}(a,k_0),~ p\|A - M \|_F^2\leq 2t}
	\sum_{i,j} p (A_{ij}-M_{ij})^2 \leq  2(2a)^{2q-2}U^2t. \end{align*}
	Therefore, using our previous estimate in \eqref{Generalbound: Expectationbound2}  for $\mathbb{E}Z(t)$ as well, we have for the variance proxy $V_{m_1 m_2}$ that $$V_{m_1m_2}
	\leq 2(2a)^{2q-2}U^2t+31(2a)^{q}U t/(8 \sqrt{z}).$$
	Hence, using \eqref{Generalbound: Expectationbound2} and
	Talagrand's inequality, we obtain
	\begin{align*}
	\mathbb{P}\left ( Z(t)  > \frac{t}{\sqrt{z}}\right ) \leq  	\mathbb{P}\left ( Z(t) - \mathbb{E} Z(t)  > \frac{ t}{32\sqrt{z}}\right ) \leq \exp \left ( \frac{-t}{
		32^2 ( 8(2a)^{2q-2}U^2z + 505(2a)^{q}U \sqrt{z}/32 )
	} \right ).
	\end{align*}
\end{proof}

\subsection{An oracle estimator in the Bernoulli model} \label{OracleBernoulli}
Here we prove that the soft-thresholding estimator proposed by Koltchinskii et. al. \cite{Koltchinskiietal} for the trace-regression setting fulfills the oracle inequality \eqref{Oracleinequality} in the Bernoulli model. \\
Their estimator is defined as
\begin{equation} \label{HardThresh}
\hat M \in \argmin_{A \in \mathbb{R}^{m_1 \times m_2}} \left (\frac{\|A \|_F^2}{m_1 m_2} - \frac{2}{n} \left \langle  Y, A \right \rangle + \lambda \| A \|_* \right )
\end{equation}
where $\lambda$ is a tuning parameter which we choose as
\begin{equation} \label{lambda} \lambda = 3 \left ( \frac{3\sqrt{2} \sigma  +  \sqrt{2C}U }{\sqrt{mn}}  \right ) \end{equation}
where $C > 0$ is the constant in Corollary 3.12 in \cite{Bandeira}.
\begin{Proposition} \label{OracleBernoulliThm} Consider the Bernoulli model \eqref{Bernoulli_model}. Assume $n \geq m \log (d) $ and that Assumption \ref{noise_boundedTR} is fulfilled. 
	Let $\hat M$ be given as in \eqref{HardThresh} with a choice of $\lambda$ as in \eqref{lambda}. Then, with $\mathbb{P}_{M_0, \sigma}$-probability of at least $1-1/d$ we have for any $M_0 \in \mathcal{A}{(a,m)}$ that 
	\begin{align*}
	\frac{\| \hat M - M_0\|_F^2}{m_1 m_2 } \leq&  \inf_{A \in \mathbb{R}^{m_1 \times m_2}} \left ( \frac{\| M_0  - A \|_F^2}{m_1 m_2} + C\frac{d \rank(A)}{n}\right ) \\
	\leq & \inf_{k \in \{0,...,m\}} \left ( \frac{\| M_0 - \mathcal{A}(a,k) \|_F^2}{m_1m_2} + C \frac{dk}{n}\right )
	\end{align*}
	for a constant $C=C(a, \sigma, U) > 0$.
\end{Proposition}
\begin{proof}
	Going through the proof of Theorem 2 and Corollary 2 in \cite{Koltchinskiietal} line by line we see that we only need to bound the spectral norm of the matrix
	$$\Sigma:= \frac{1}{n} \left ( B_{ij} \epsilon_{ij} \right )_{i,j} $$
	by $\lambda/3$ with high probability. Using self-adjoint dilation to generalize Corollary 3.12 and Remark 3.13 in \cite{Bandeira} for rectangular matrices (with choices $\eps=1/2$, $\tilde{\sigma_*}=U$ and $$\tilde{\sigma}=\max \left ( \max_j \sqrt{ \sum_{i=1}^{m_1} \mathbb{E}_\sigma B_{ij}^2 \epsilon_{ij}^2 }, \max_i\sqrt{ \sum_{j=1}^{m_2} \mathbb{E}_\sigma B_{ij}^2 \epsilon_{ij}^2 } \right ) = \sigma\sqrt{ n/m}$$  there) we obtain 
	$$ \mathbb{P}_\sigma \left ( \left \| \sum_{i=1}^n \eps_{i}X_{i} \right \| > 3\sqrt{2} \sigma \sqrt{\frac{n}{m}} + t \right ) \leq d \exp \left (-\frac{-t^2}{C_1U^2} \right)$$
	for a constant $C_1 > 0$. Choosing $t=\sqrt{2C_1} U\sqrt{\frac{n}{m}}$ and using that $n \geq m \log(d)$ yields that $\Xi$ occurs with $\mathbb{P}_\sigma$-probability at least $1-1/d$.
\end{proof}

\section*{Acknowledgements}  The work of A.~Carpentier is supported by the DFG’s Emmy Noether grant MuSyAD (CA 1488/1-1). The work of O. Klopp was conducted as part of the project Labex MME-DII (ANR11-LBX-0023-01). The work of M. L\"offler was supported by the UK Engineering and Physical Sciences Research Council (EPSRC) grant EP/L016516/1 and the European Research Council (ERC) grant No. 647812. The latter ERC grant also supported R. Nickl, who is further grateful to A. Tsybakov and ENSAE Paris for their hospitality during a visit in April 2016 where part of this research was undertaken.

\end{document}